\documentclass[a4paper,12pt]{article}

\usepackage[english]{babel}
\usepackage[T1]{fontenc}

\usepackage{tikz}
\tikzstyle{vertex}=[circle, draw, inner sep=2pt, minimum size=6pt]

\usepackage{mathrsfs}

\usepackage[a4paper,top=3cm,bottom=2cm,left=3cm,right=3cm,marginparwidth=1.75cm]{geometry}

\usepackage{comment}
\usepackage{amsfonts}
\usepackage{fullpage}
\usepackage{latexsym}
\usepackage{blkarray}
\usepackage{graphicx}
\usepackage{float}
\usepackage{enumerate}
\usepackage{amsthm,amsmath,amssymb}
\usepackage{verbatim}
\usepackage{enumitem}
\usepackage{mathtools}
\usepackage{algorithm}
\usepackage[noend]{algpseudocode}

\providecommand{\keywords}[1]{
  \small	
  \textbf{\textit{Keywords---}} #1
}

\newtheorem{lemma}{Lemma}[section]
\newtheorem{corollary}{Corollary}[section]
\newtheorem{proposition}{Proposition}[section]

\newtheorem{example}{Example}[section]

\newcommand{\nc}{\newcommand}
\newcommand{\Yildirim}{Y{\i}ld{\i}r{\i}m}

\nc{\cA}{{\cal A}}
\nc{\cB}{{\cal B}}
\nc{\cC}{{\cal C}}
\nc{\cD}{{\cal D}}
\nc{\cE}{{\cal E}}
\nc{\cG}{{\cal G}}
\nc{\cF}{{\cal F}}
\nc{\cH}{{\cal H}}
\nc{\cI}{{\cal I}}
\nc{\cK}{{\cal K}}
\nc{\cL}{{\cal L}}
\nc{\cM}{{\cal M}}
\nc{\cN}{{\cal N}}
\nc{\cO}{{\cal O}}
\nc{\cP}{{\cal P}}
\nc{\cQ}{{\cal Q}}
\nc{\cR}{{\cal R}}
\nc{\cS}{{\cal S}}
\nc{\cT}{{\cal T}}
\nc{\cV}{{\cal V}}
\nc{\tx}{{\tilde x}}
\nc{\la}{{\langle}}
\nc{\ra}{{\rangle}}
\nc{\ts}{\textsuperscript}
\def\R{\mathbb{R}}

\title{On Exact and Inexact RLT and SDP-RLT Relaxations of Quadratic Programs with Box Constraints}
\author{Yuzhou Qiu\thanks{School of Mathematics, Peter Guthrie Tait Road, The University of Edinburgh, Edinburgh, EH9 3FD, United Kingdom. E-mail: \tt{y.qiu-16@sms.ed.ac.uk}} \and E. Alper \Yildirim\thanks{School of Mathematics, Peter Guthrie Tait Road, The University of Edinburgh, Edinburgh, EH9 3FD, United Kingdom. ORCID ID: 0000-0003-4141-3189 E-mail: \tt{E.A.Yildirim@ed.ac.uk}}}
\date{1 March 2023}
\begin{document}

\maketitle
\begin{abstract}
Quadratic programs with box constraints involve minimizing a possibly nonconvex quadratic function subject to lower and upper bounds on each variable. This is a well-known NP-hard problem that frequently arises in various applications. We focus on two convex relaxations, namely the RLT (Reformulation-Linearization Technique) relaxation and the SDP-RLT relaxation obtained by adding semidefinite constraints to the RLT relaxation. Both relaxations yield lower bounds on the optimal value of a quadratic program with box constraints. We present complete algebraic descriptions of the set of instances that admit exact RLT relaxations as well as those that admit exact SDP-RLT relaxations. We show that our 
descriptions can be converted into algorithms for efficiently constructing instances with exact or inexact relaxations.
\end{abstract}

\keywords{Quadratic programming with box constraints, Reformulation linearization technique, Semidefinite relaxation, Convex underestimator}

{\bf AMS Subject Classification:} 90C20, 90C22, 90C26

\section{Introduction} \label{intro}

A quadratic program with box constraints is an optimization problem in which a possibly nonconvex quadratic function is minimized subject to lower and upper bounds on each variable:
\[
\textrm{(BoxQP)} \quad \ell^* = \min\limits_{x \in \R^n} \left\{q(x): x \in F\right\},
\]
where 
$q: \R^n \to \R$ and $F \subseteq \R^n$ are respectively given by
\[
q(x) = \textstyle \frac{1}{2} x^T Q x + c^T x, \quad F = \left\{x \in \R^n: 0 \leq x \leq e \right\}.
\]
Here, $e \in \R^n$ denotes the vector of all ones. The parameters of the problem are given by the pair $(Q,c) \in \cS^n \times \R^n$, where $\cS^n$ denotes the set of $n \times n$ real symmetric matrices. The optimal value is denoted by $\ell^* \in \R$. Note that any quadratic program with finite lower and upper bounds on each variable can be easily transformed into the above form.  

(BoxQP) is regarded as a ``fundamental problem'' in global optimization that appears in a multitude of applications (see, e.g., \cite{angelis1997quadratic,HPT2000}). If $Q$ is a positive semidefinite matrix, (BoxQP) can be solved in polynomial time~\cite{kozlov1979polynomial}. However, if $Q$ is an indefinite or negative semidefinite matrix, then (BoxQP) is an NP-hard problem \cite{Sahni74,Pardalos1991QuadraticPW}. In fact, it is even NP-hard to approximate a local minimizer of (BoxQP)~\cite{AhmadiZ22}.

\subsection{RLT and SDP-RLT Relaxations}

By using a simple ``lifting'' idea, (BoxQP) can be equivalently reformulated as 
\[
\textrm{(L-BoxQP)} \quad \ell^* = \min\limits_{(x,X) \in \cF}\textstyle \frac{1}{2}\langle Q, X \rangle + c^T x,
\]
where $\langle A, B \rangle = \textrm{trace}(A^T B) = \sum\limits_{i=1}^p \sum\limits_{j=1}^q A_{ij} B_{ij}$ for any $A \in \R^{p \times q}$ and $B \in \R^{p \times q}$, and 
\begin{equation} \label{def_cF}
\cF = \left\{(x,X) \in \R^n \times \cS^n: 0 \leq x \leq e, \quad X_{ij} = x_i x_j, \quad 1 \leq i \leq j \leq n\right\}.
\end{equation}
Since (L-BoxQP) is an optimization problem with a linear objective function over a nonconvex feasible region, one can replace $\cF$ by $\textrm{conv}(\cF)$, where $\textrm{conv}(\cdot)$ denotes the convex hull, without affecting the optimal value. Many convex relaxations of (BoxQP) arise from this reformulation by employing outer approximations of $\textrm{conv}(\cF)$ using tractable convex sets. 

A well-known relaxation of $\textrm{conv}(\cF)$ is obtained by replacing the nonlinear equalities $X_{ij} = x_i x_j$ by the so-called McCormick inequalities~\cite{mccormick1976computability}, which gives rise to the RLT (Reformulation-Linearization Technique) relaxation of (BoxQP) (see, e.g., \cite{SheraliT92}):
\[
\textrm{(R)} \quad \ell^*_R = \min\limits_{(x,X) \in \R^n \times \cS^n} \left\{\frac{1}{2}\langle Q, X\rangle + c^T x: (x,X) \in \cF_R\right\},
\]
where 
\begin{equation} \label{def_cFR_alt}
\cF_R = \left\{(x,X) \in \R^n \times \cS^n:\begin{array}{rcccl} 0 & \leq & x & \leq & e\\ \max\{x_i + x_j -1,0\} & \leq & X_{ij} & \leq & \min \{x_i,x_j\}, \quad 1 \leq i \leq j \leq n \end{array}\right\}.
\end{equation}

The RLT relaxation (R) of (BoxQP) can be further strengthened by adding tighter semidefinite constraints~\cite{1987Quadratic}, giving rise to the SDP-RLT relaxation:
\[
\textrm{(RS)} \quad \ell^*_{RS} = \min\limits_{(x,X) \in \R^n \times \cS^n} \left\{\frac{1}{2}\langle Q, X\rangle + c^T x: (x,X) \in \cF_{RS}\right\},
\]
where 
\begin{equation} \label{def_cFRS}
\cF_{RS} = \left\{(x,X) \in \R^n \times \cS^n: (x,X) \in \cF_R, \quad X - x x^T \succeq 0\right\}.
\end{equation}

The optimal value of each of the RLT and SDP-RLT relaxations, denoted by $\ell^*_R$ and $\ell^*_{RS}$, respectively, yields a lower bound on the optimal value of (BoxQP). The SDP-RLT relaxation is clearly at least as tight as the RLT relaxation, i.e., 
\begin{equation} \label{convex_rels}
\ell^*_R \leq \ell^*_{RS} \leq \ell^*. 
\end{equation}

\subsection{Motivation and Contributions}

Convex relaxations play a fundamental role in the design of global solution methods for nonconvex optimization problems. In particular, one of the most prominent algorithmic approaches for globally solving nonconvex optimization problems is based on a branch-and-bound framework, in which the feasible region is systematically subdivided into smaller subregions and a sequence of subproblems is solved to obtain increasingly tighter lower and upper bounds on the optimal value in each subregion. The lower bounds in such a scheme are typically obtained by solving a convex relaxation. For instance, several well-known optimization solvers such as {\tt ANTIGONE}~\cite{misener2014antigone}, {\tt BARON}~\cite{ts:05}, {\tt CPLEX}~\cite{cplex}, and {\tt GUROBI}~\cite{gurobi} utilize convex relaxations for globally solving nonconvex quadratic programs.

In this paper, our main goal is to describe the set of instances of (BoxQP) that admit exact RLT relaxations (i.e., $\ell_R^* = \ell^*$) as well as those that admit exact SDP-RLT relaxations (i.e., $\ell_{RS}^* = \ell^*$). Such descriptions shed light on easier subclasses of a difficult optimization problem. In addition, we aim to develop efficient algorithms for constructing an instance of (BoxQP) that admits an exact or inexact relaxation. Such algorithms can be quite useful in computational experiments for generating instances of (BoxQP) for which a particular relaxation will have a predetermined exactness or inexactness guarantee.

Our contributions are as follows.

\begin{enumerate}
    \item By utilizing the recently proposed perspective on convex underestimators induced by convex relaxations~\cite{yildirim2020alternative}, we  
    establish several useful properties of each of the two convex underestimators associated with the RLT relaxation and the SDP-RLT relaxation. 
    \item We present two equivalent algebraic descriptions of the set of instances of (BoxQP) that admit exact RLT relaxations. The first description arises from the analysis of the convex underestimator induced by the RLT relaxation, whereas the second description is obtained by using linear programming duality. 
    \item By relying on the second description of the set of instances with an exact RLT relaxation, we propose an algorithm for efficiently constructing an instance of (BoxQP) that admits an exact RLT relaxation and another algorithm for constructing an instance with an inexact RLT relaxation.
    \item We establish that strong duality holds and that primal and dual optimal solutions are attained for the SDP-RLT relaxation and its dual. By relying on this relation, we give an algebraic description of the set of instances of (BoxQP) that admit an exact SDP-RLT relaxation.
    \item By utilizing this algebraic description, we propose an algorithm for constructing an instance of (BoxQP) that admits an exact SDP-RLT relaxation and another one for constructing an instance that admits an exact SDP-RLT relaxation but an inexact RLT relaxation.
\end{enumerate}

This paper is organized as follows. We briefly review the literature in Section~\ref{lit_rev} and define our notation in Section~\ref{notation}. We review the optimality conditions in Section~\ref{Sec2}. We present several properties of the convex underestimators arising from the RLT and SDP-RLT relaxations in Section~\ref{Sec3}. Section~\ref{Sec4} focuses on the description of instances with exact RLT relaxations and presents two algorithms for constructing instances with exact and inexact RLT relaxations. SDP-RLT relaxations are treated in Section~\ref{Sec5}, which includes an algebraic description of instances with exact SDP-RLT relaxations and two algorithms for constructing instances with different exactness guarantees. We present several numerical examples and a brief discussion in Section~\ref{Sec6}. Finally, Section~\ref{Sec7} concludes the paper.

\subsection{Literature Review} \label{lit_rev}

Quadratic programs with box constraints have been extensively studied in the literature. Since our focus is on convex relaxations in this paper, we will mainly restrict our review accordingly. 

The set $\textrm{conv}(\cF)$, where $\cF$ is given by \eqref{def_cF}, has been investigated in several papers (see, e.g., ~\cite{yajima1998polyhedral,burer2009nonconvex,anstreicher2009semidefinite,anstreicher2010computable,DongL13,BonamiGL18}). This is a nonpolyhedral convex set even for $n = 1$. However, it turns out that $\textrm{conv}(\cF)$ is closely related to the so-called Boolean quadric polytope~\cite{padberg1989boolean} that arises in unconstrained binary quadratic programs, which can be formulated as an instance of (BoxQP)~\cite{Rag1969}, and is given by $\textrm{conv}(\cF^-)$, where
\begin{equation} \label{def_cF-}
\cF^- = \left\{(x,z) \in \R^n \times \R^{{n \choose 2}}: x_i \in \{0,1\}, \quad  z_{ij} = x_i x_j, \quad 1 \leq i < j \leq n\right\}.
\end{equation}
The linear programming relaxation of $\textrm{conv}(\cF^-)$, denoted by $\cF_R^-$, is given by
\begin{equation} \label{def_cFR-}
\cF_R^- = \left\{(x,z) \in \R^n \times \R^{{n \choose 2}}:\begin{array}{rcccl} 0 & \leq & x & \leq & e\\ \max\{x_i + x_j -1,0\} & \leq & z_{ij} & \leq & \min \{x_i,x_j\}, \quad 1 \leq i < j \leq n \end{array}\right\},  
\end{equation}
which is very similar to $\cF_R$, except that McCormick inequalities are only applied to $1 \leq i < j \leq n$. In particular, $\cF_R^- = \textrm{conv}(\cF^-)$ for $n = 2$~\cite{mccormick1976computability}. Padberg~\cite{padberg1989boolean} identifies several facets of $\textrm{conv}(\cF^-)$ and shows that the components of each vertex of $\cF_R^-$ are in the set $\{0,\textstyle \frac{1}{2},1\}$. Yajima and Fujie~\cite{yajima1998polyhedral} show how to extend the valid inequalities for $\cF^-$ in~\cite{padberg1989boolean} to $\cF$. Burer and Letchford~\cite{burer2009nonconvex} extend this result further by observing that $\textrm{conv}(\cF^-)$ is the projection of $\textrm{conv}(\cF)$ onto the ``common variables.'' They also give a description of the set of extreme points of $\textrm{conv}(\cF)$. We refer the reader to~\cite{DongL13} for further refinements and to~\cite{BonamiGL18} for a computational procedure based on such valid inequalities. 

Anstreicher~\cite{anstreicher2009semidefinite} reports computational results illustrating that the SDP-RLT relaxation significantly improves the RLT relaxation and gives a theoretical justification of the improvement by comparing $\cF_R$ and $\cF_{RS}$ for $n = 2$. Anstreicher and Burer~\cite{anstreicher2010computable} show that $\cF_{RS} = \textrm{conv}(\cF)$ if and only if $n \leq 2$. In particular, this implies that the SDP-RLT relaxation of (BoxQP) is always exact for $n \leq 2$.

We next briefly review the literature on exact convex relaxations. Several papers have identified conditions under which a particular convex relaxation of a class of optimization problems is exact. For quadratically constrained quadratic programs, we refer the reader to~\cite{KimK03,SojoudiL14,JeyakumarL14,Locatelli16,Ho-NguyenK17,BurerY20,WangK22a,AzumaFKY22} for various exactness conditions for second-order cone or semidefinite relaxations. Recently, a large family of convex relaxations of general quadratic programs was considered in a unified manner through induced convex underestimators and a general algorithmic procedure was proposed for constructing instances with inexact relaxations for various convex relaxations~\cite{yildirim2020alternative}.

In this work, our focus is on algebraic descriptions and algorithmic constructions of instances of (BoxQP) that admit exact and inexact RLT and SDP-RLT relaxations. Therefore, our focus is similar to~\cite{SagolY15,GokmenY22}, which presented descriptions of such instances of standard quadratic programs for RLT and SDP-RLT relaxations, respectively.

\subsection{Notation} \label{notation}

We use $\R^n$, $\R^{m \times n}$, and $\cS^n$ to denote the $n$-dimensional Euclidean space, the set of $m \times n$ real matrices, and the space of $n \times n$ real symmetric matrices, respectively.  We use 0 to denote the real number 0, the vector of all zeroes, as well as the matrix of all zeroes, which should always be clear from the context. We reserve $e \in \R^n$ for the vector of all ones. All inequalities on vectors or matrices are componentwise. For $A \in \cS^n$, we use $A \succeq 0$ (resp., $A \succ 0$) to denote that $A$ is positive semidefinite (resp., positive definite). For index sets $\mathbb{J} \subseteq \{1,\ldots,m\}$, $\mathbb{K} \subseteq \{1,\ldots,n\}$, $x \in \R^n$, and $B \in \R^{m \times n}$, we denote by $x_\mathbb{K} \in \R^{|\mathbb{K}|}$ the subvector of $x$ restricted to the indices in $\mathbb{K}$ and by $B_{\mathbb{J}\mathbb{K}} \in \R^{|\mathbb{J}|\times|\mathbb{K}|}$ the submatrix of $B$ whose rows and columns are indexed by $\mathbb{J}$ and $\mathbb{K}$, respectively, where $|\cdot|$ denotes the cardinality of a finite set. We simply use $x_j$ and $Q_{ij}$ for singleton index sets. For any $U \in \R^{m \times n}$ and $V \in \R^{m \times n}$, the trace inner product is denoted by 
\[
\langle U, V \rangle = \textrm{trace}(U^T V) = \sum\limits_{i=1}^m \sum\limits_{j = 1}^n U_{ij} V_{ij}.
\]

For an instance of (BoxQP) given by $(Q,c) \in \cS^n \times \R^n$, we define
\begin{eqnarray}
q(x) & = & \textstyle \frac{1}{2} x^T Q x + c^T x, \label{def_qx}\\
F & = & \{x \in \R^n: 0 \leq x \leq e\}, \label{feasible set of box QP}\\
\ell^* & = & \min\limits_{x \in \R^n} \left\{q(x): x \in F\right\}, \label{def_ell*}\\
V & = & \{x \in F: x_j \in \{0,1\}, \quad j = 1, \ldots, n\}. \label{vertices of box QP}
\end{eqnarray}

For a given instance of (BoxQP), note that $q(x)$, $F$, $\ell^*$, and $V$ denote the objective function, the feasible region, the optimal value, and the set of vertices, respectively. For $\hat x \in F$, we define the following index sets:
\begin{eqnarray}
\mathbb{L} & = & \mathbb{L}(\hat x) = \{j \in \{1,\ldots,n\}: \hat x_j = 0\}, \label{def_indL} \\
\mathbb{B} & = & \mathbb{B}(\hat x) = \{j \in \{1,\ldots,n\}: 0 < \hat x_j < 1\}, \label{def_indB} \\
\mathbb{U} & = & \mathbb{U}(\hat x) = \{j \in \{1,\ldots,n\}: \hat x_j = 1\}. \label{def_indU}
\end{eqnarray}

\section{Optimality Conditions} \label{Sec2}

In this section, we review first-order and second-order optimality conditions for (BoxQP). 

Let $\hat x \in F$ be a local minimizer of (BoxQP).
By the first-order optimality conditions, there exists $(\hat u, \hat v) \in \R^n \times \R^n$ such that 
\begin{eqnarray} \label{kkt_cond}
        Q \hat x + c + \hat u - \hat v & = & 0, \label{kkt1}\\
        \hat u_{\mathbb{L} \cup \mathbb{B}} & = & 0, \label{kkt2}\\
        \hat v_{\mathbb{B} \cup \mathbb{U}} & = & 0, \label{kkt3}\\
        \hat u & \geq & 0, \label{kkt4}\\
        \hat v & \geq & 0. \label{kkt5}
\end{eqnarray}
Note that $\hat u \in \R^n$ and $\hat v \in \R^n$ are the Lagrange multipliers corresponding to the constraints $x \leq e$ and $x \geq 0$ in (BoxQP), respectively. 

For a local minimizer $\hat x \in F$ of (BoxQP), the second-order optimality conditions are given by 
\begin{equation} \label{sonc}
d^TQd \geq 0, \quad \forall \, d \in D(\hat x),
\end{equation}
where $D(\hat x)$ is the set of feasible directions at $\hat x$ at which the directional derivative of the objective function vanishes, i.e., 
\begin{equation} \label{def_Dx}
D(\hat x) := \{d \in \R^n: (Q \hat x + c)^T d = 0, \quad d_\mathbb{L} \geq 0, \quad d_\mathbb{U} \leq 0\}.
\end{equation}
Note, in particular, that 
\begin{equation} \label{Q_BB_psd}
\textrm{$\hat x \in F$ is a local minimizer} \Rightarrow Q_{\mathbb{B}\mathbb{B}} \succeq 0.   
\end{equation}
In fact, $\hat x \in F$ is a local minimizer of (BoxQP) if and only if the first-order and second-order optimality conditions given by \eqref{kkt1}--\eqref{kkt5} and \eqref{sonc}, respectively, are satisfied (see, e.g., \cite{ref:majthay1971,ref:jiaquan1982}). 

\section{Properties of RLT and SDP-RLT Relaxations} \label{Sec3}

Given an instance of (BoxQP), recall that the RLT relaxation is given by
\[
\textrm{(R)} \quad \ell^*_R = \min\limits_{(x,X) \in \R^n \times \cS^n} \left\{\frac{1}{2}\langle Q, X\rangle + c^T x: (x,X) \in \cF_R\right\},
\]
where $\cF_R$ is given by \eqref{def_cFR_alt}, and the SDP-RLT relaxation by 
\[
\textrm{(RS)} \quad \ell^*_{RS} = \min\limits_{(x,X) \in \R^n \times \cS^n} \left\{\frac{1}{2}\langle Q, X\rangle + c^T x: (x,X) \in \cF_{RS}\right\},
\]
where $\cF_{RS}$ is given by \eqref{def_cFRS}. 

Every convex relaxation of a nonconvex quadratic program obtained through lifting induces a convex underestimator on the objective function over the feasible region~\cite{yildirim2020alternative}. In this section, we introduce the convex underestimators induced by RLT and SDP-RLT relaxations and establish several properties of these underestimators. 

\subsection{Convex Underestimators} \label{convex_underest}

In this section, we introduce the convex underestimators induced by the RLT and SDP-RLT relaxations. Let us first define the following sets parametrized by $\hat x \in F$:
\begin{eqnarray} \label{param_feas_regions}
\cF_R(\hat x) & = & \{(x,X) \in \cF_{R}: x = \hat x\}, \quad \hat x \in F, \label{param_feas_RLT}\\
\cF_{RS}(\hat x) & = & \{(x,X) \in \cF_{RS}: x = \hat x\}, \quad \hat x \in F. \label{param_feas_SDP-RLT}
\end{eqnarray}

For each $\hat x \in F$, we clearly have $\{(\hat x,\hat x \hat x^T)\} \subseteq \cF_{RS}(\hat x) \subseteq \cF_R(\hat x)$ and 
\[
\bigcup\limits_{\hat x \in F} \cF_R(\hat x) = \cF_R, \quad \bigcup\limits_{\hat x \in F} \cF_{RS}(\hat x) = \cF_{RS}.
\]

Next, we define the following functions:
\begin{eqnarray} \label{conv_underest}
    \ell_R(\hat x) & = & \min_{x \in \R^n,X \in \cS^n} \left\{ \textstyle \frac{1}{2}\langle Q, X \rangle + c^T x: (x,X) \in \cF_R(\hat x)\right\}, \quad \hat x \in F, \label{conv_und_RLT}\\
    \ell_{RS}(\hat x) & = & \min_{x \in \R^n,X \in \cS^n} \left\{ \textstyle \frac{1}{2}\langle Q, X \rangle + c^T x: (x,X) \in \cF_{RS}(\hat x)\right\}, \quad \hat x \in F. \label{conv_und_SDP-RLT}
\end{eqnarray}
Note that the functions $\ell_R(\cdot)$ and $\ell_{RS}(\cdot)$ return the best objective function value of the corresponding relaxation subject to the additional constraint that $x = \hat x$. By~\cite{yildirim2020alternative}, each of $\ell_R(\cdot)$ and $\ell_{RS}(\cdot)$ is a convex function over $F$ satisfying the relations 
\begin{equation} \label{conv_underest_rel_1}
\ell_R(\hat x) \leq \ell_{RS}(\hat x) \leq q(\hat x), \quad \hat x \in F,
\end{equation}
and 
\begin{eqnarray} \label{conv_underest_rel_2}
\textrm{(R1)} \quad \ell^*_R & = & \min\limits_{x \in F} \ell_R(x), \label{def_R1} \\
\textrm{(RS1)} \quad \ell^*_{RS} & = & \min\limits_{x \in F} \ell_{RS}(x). \label{def_RS1}
\end{eqnarray}
The convex underestimators $\ell_R(\cdot)$ and $\ell_{RS}(\cdot)$ allow us to view the RLT and SDP-RLT relaxations in the original space $\R^n$ of (BoxQP) by appropriately projecting out the lifted variables $X \in \cS^n$ that appear in each of (R) and (RS). As such, (R1) and (RS1) can be viewed as ``reduced'' formulations of the RLT relaxation and the SDP-RLT relaxation, respectively. 
In the remainder of this manuscript, we will alternate between the two equivalent formulations (R) and (R1) for the RLT relaxation as well as (RS) and (RS1) for the SDP-RLT relaxation.

\subsection{Properties of Convex Underestimators}

In this section, we present several properties of the convex underestimators $\ell_R(\cdot)$ and $\ell_{RS}(\cdot)$ given by \eqref{conv_und_RLT} and \eqref{conv_und_SDP-RLT}, respectively. 

First, we start with the observation that $\ell_R(\cdot)$ has a very specific structure with a simple closed-form expression.

\begin{lemma} \label{piecewise-lin}
$\ell_R(\cdot)$ is a piecewise linear convex function on $F$ given by 
\begin{equation} \label{ell_R_closed}
\ell_R(\hat x) = \frac{1}{2}\left(\sum\limits_{(i,j): Q_{ij} > 0} Q_{ij} \max\{0, \hat{x}_i + \hat{x}_j - 1\} + \sum\limits_{(i,j): Q_{ij} < 0} Q_{ij} \min\{\hat{x}_i,\hat{x}_j\}\right) + c^T \hat x, \quad \hat x \in F.    
\end{equation}
\end{lemma}
\begin{proof}
For each $\hat x \in F$, the relation \eqref{ell_R_closed} follows from \eqref{conv_und_RLT} and \eqref{def_cFR_alt}.
It follows that $\ell_R(\cdot)$ is a piecewise linear convex function on $F$ since it is given by the sum of a finite number of piecewise linear convex functions.
\end{proof}

In contrast with $\ell_R(\cdot)$ given by the optimal value of a simple linear programming problem with bound constraints, $\ell_{RS}(\cdot)$ does not, in general, have a simple closed-form expression as it is given by the optimal value of a semidefinite programming problem. 

The next result states a useful decomposition property regarding the sets $\cF_R(\hat x)$ and $\cF_{RS}(\hat x)$.

\begin{lemma} \label{feas_cond}
For any $\hat x \in F$,
$(\hat x, \hat X) \in \cF_R(\hat x)$ if and only if there exists $\hat M \in \cM_R(\hat x)$ such that $\hat X = \hat x \hat x^T + \hat M$, where
\begin{equation} \label{def_cMR}
\cM_R(\hat x) = \left\{M \in \cS^n: \begin{array}{rcl} M_{ij} & \leq & \min\{\hat x_i - \hat x_i \hat x_j, \hat x_j - \hat x_i \hat x_j\}, \quad i \in \mathbb{B}, j \in \mathbb{B},\\
M_{ij} & \geq & \max\{-\hat x_i \hat x_j, \hat x_i + \hat x_j - 1 - \hat x_i \hat x_j\}, \quad i \in \mathbb{B}, j \in \mathbb{B},\\
M_{ij} & = & 0, \quad \textrm{otherwise.}
\end{array}
\right\}, 
\end{equation}
where $\mathbb{B}$ is given by \eqref{def_indB}. 
Furthermore, $(\hat x, \hat X) \in \cF_{RS}(\hat x)$ if and only if $\hat M \in \cM_{RS}(\hat x)$, where 
\begin{equation} \label{def_cMRS}
\cM_{RS}(\hat x) = \left\{M \in \cS^n: M \in \cM_R(\hat x), \quad M \succeq 0\right\}. 
\end{equation}
\end{lemma}
\begin{proof}
Both assertions follow from \eqref{param_feas_RLT}, \eqref{def_cFR_alt}, \eqref{param_feas_SDP-RLT}, \eqref{def_cFRS}, and the decomposition $\hat X = \hat x \hat x^T + \hat M$.     
\end{proof}

By Lemma~\ref{feas_cond}, we remark that $M_{ij}$ has a negative lower bound and a positive upper bound in \eqref{def_cMR} if and only if $i \in \mathbb{B}$ and $j \in \mathbb{B}$. Therefore, for any $\hat x \in \cF$ and any $(\hat x,\hat X) \in \cF_R$ (and hence any $(\hat x,\hat X) \in \cF_{RS}$), we obtain
\begin{equation} \label{X_fixed_part}
 \hat X_{ij} = \hat x_i \hat x_j, \quad i \not \in \mathbb{B}, \textrm{ or } j \not \in \mathbb{B}.   
\end{equation}
This observation yields the following result.

\begin{corollary} \label{vertex-prop}
For any vertex $v \in F$, $\cF_R(v) = \cF_{RS}(v) = \{(v,vv^T)\}$.    
\end{corollary}
\begin{proof}
The claim directly follows from \eqref{X_fixed_part} since $\mathbb{B} = \emptyset$.   
\end{proof}

The decomposition in Lemma~\ref{feas_cond} can be translated into the functions $\ell_R(\cdot)$ and $\ell_{RS}(\cdot)$.

\begin{lemma} \label{underest_decom}
For each $\hat x \in F$,
\begin{eqnarray} \label{decomp_rels}
\ell_R(\hat x) & = & q(\hat x) + \textstyle \frac{1}{2} \min\limits_{M \in \cM_R(\hat x)} \langle Q, M \rangle, \label{decomp_rlt}\\
\ell_{RS}(\hat x) & = & q(\hat x) + \textstyle \frac{1}{2} \min\limits_{M \in \cM_{RS}(\hat x)} \langle Q, M \rangle, \label{decomp_rlt_sdp}
\end{eqnarray}
where $\cM_R(\hat x)$ and $\cM_{RS}(\hat x)$ are given by \eqref{def_cMR} and \eqref{def_cMRS}, respectively. 
\end{lemma}
\begin{proof}
The assertions directly follow from \eqref{conv_und_RLT}, \eqref{conv_und_SDP-RLT}, and Lemma~\ref{feas_cond}.    
\end{proof}

By Lemma~\ref{underest_decom}, we can easily establish the following properties.

\begin{lemma} \label{exact_underests}
Let $\hat x \in F$ and let $\mathbb{B} = \mathbb{B}(\hat x)$, where $\mathbb{B}(\hat x)$ is given by \eqref{def_indB}. 
\begin{enumerate}
    \item[(i)] $\ell_R(\hat x) = q(\hat x)$ if and only if $\hat x$ is a vertex of $F$ or $Q_{\mathbb{B}\mathbb{B}} = 0$.
    \item[(ii)] $\ell_{RS}(\hat x) = q(\hat x)$ if and only if $\hat x$ is a vertex of $F$ or $Q_{\mathbb{B}\mathbb{B}} \succeq 0$.
\end{enumerate}
\end{lemma}
\begin{proof}
By Lemma~\ref{underest_decom}, $\ell_R(\hat x) = q(\hat x)$ (resp., $\ell_{RS}(\hat x) = q(\hat x)$) if and only if $\min\limits_{M \in \cM_R(\hat x)} \langle Q, M \rangle = 0$ (resp., $\min\limits_{M \in \cM_{RS}(\hat x)} \langle Q, M \rangle = 0$). The assertions now follow from Lemma~\ref{feas_cond}.   
\end{proof}

Lemma~\ref{exact_underests} immediately gives rise to the following results about the underestimator $\ell_{RS}(\cdot)$.

\begin{corollary} \label{implications}
\begin{enumerate}
    \item[(i)] If $Q \succeq 0$, then $\ell_{RS}(\hat x) = q(\hat x)$ for each $\hat x \in F$.
    \item[(ii)] For any local or global minimizer $\hat x \in F$ of (BoxQP), we have $\ell_{RS}(\hat x) = q(\hat x)$. 
\end{enumerate}
\end{corollary}
\begin{proof}
If $Q \succeq 0$, then $Q_{\mathbb{B}\mathbb{B}} \succeq 0$ for each $\mathbb{B} \subseteq \{1,\ldots,n\}$. Therefore, both assertions follow from Lemma~\ref{exact_underests}(ii) since $Q_{\mathbb{B}\mathbb{B}} \succeq 0$ at any local or global minimizer of (BoxQP) by \eqref{Q_BB_psd}.  
\end{proof}

Corollary~\ref{implications}(i) in fact holds for SDP relaxations of general quadratic programs and a result similar to Corollary~\ref{implications}(ii) was established for general quadratic programs with a bounded feasible region~\cite{yildirim2020alternative}. We remark that Corollary~\ref{implications}(ii) presents a desirable property of the SDP-RLT relaxation, which is a necessary condition for its exactness by \eqref{conv_underest_rel_1} and \eqref{def_RS1}. However, this condition, in general, is not sufficient.

\section{Exact and Inexact RLT Relaxations} \label{Sec4}

In this section, we focus on instances of (BoxQP) that admit exact and inexact RLT relaxations. We first establish a useful property of the set of optimal solutions of RLT relaxations. Using this property, we present two equivalent but different algebraic descriptions of instances with exact RLT relaxations. By utilizing one of these descriptions, we present an algorithm for constructing instances of (BoxQP) with an exact RLT relaxation and another algorithm for constructing instances with an inexact RLT relaxation. 

\subsection{Optimal Solutions of RLT Relaxations}

In this section, we present useful properties of the set of optimal solutions of RLT relaxations. Our first result establishes the existence of a minimizer of (R1) with a very specific structure. 

\begin{proposition} \label{half_frac}
For the RLT relaxation of any instance of (BoxQP), there exists an optimal solution $\hat x \in F$ of (R1), where (R1) is given by \eqref{def_R1}, such that  
$\hat x_j \in \{0,\textstyle \frac{1}{2},1\}$ for each $j = 1,\ldots,n$.    
\end{proposition}
\begin{proof}
Let $\hat x \in F$ be an optimal solution of (R1), i.e.,  $\ell^*_R =\ell_R(\hat x)$. Suppose that there exists $k \in \{1,\ldots,n\}$ such that $\hat x_k \not \in \{0,\textstyle \frac{1}{2},1\}$. We will show that one can construct another $\tilde x \in F$ such that $\ell_R(\tilde x) = \ell_R(\hat x) = \ell^*_R$ and $\tilde x_j \in \{0,\textstyle \frac{1}{2},1\}$ for each $j = 1,\ldots,n$. 

Let $\alpha = \min\{\hat x_k,1 - \hat x_k\} \in (0,\textstyle \frac{1}{2})$ and let 
\begin{eqnarray*}
\alpha_l & = & \max\left\{\left(\max\limits_{j:\min\{\hat x_j,1 - \hat x_j\}<\alpha}\min\{\hat x_j,1 - \hat x_j\}\right),0\right\}, \\
\alpha_u & = & \min\left\{\left(\min\limits_{j:\min\{\hat x_j,1 - \hat x_j\}>\alpha}\min\{\hat x_j,1 - \hat x_j\}\right),\textstyle \frac{1}{2}\right\},
\end{eqnarray*}
with the usual conventions that the minimum and the maximum over the empty set are defined to be $+\infty$ and $-\infty$, respectively. Note that $0 \leq \alpha_l < \alpha < \alpha_u \leq \textstyle \frac{1}{2}$. Let us define the following index sets:
\begin{eqnarray*}
 \mathbb{I}_1 & = & \{j \in \{1,\ldots,n\}: \hat x_j = \alpha\}, \\
 \mathbb{I}_2 & = & \{j \in \{1,\ldots,n\}: \hat x_j = 1 - \alpha\}, \\
 \mathbb{I}_3 & = & \{j \in \{1,\ldots,n\}: \hat x_j \in [0,\alpha_l] \cup [\alpha_u,1 - \alpha_u] \cup [1 - \alpha_l,1]\}.
\end{eqnarray*}
Note that $\mathbb{I}_1, \mathbb{I}_2, \mathbb{I}_3$
is a partition of the index set by the definitions of $\alpha_l$ and $\alpha_u$, and we have $k \in \mathbb{I}_1 \cup \mathbb{I}_2$. 
Let us define a direction $\hat d \in \R^n$ by 
\[
\hat d_j = \left\{\begin{array}{rr} 1, & j \in \mathbb{I}_1,\\
-1, & j \in \mathbb{I}_2,\\
0, & j \in \mathbb{I}_3.
\end{array}
\right.
\]
Consider $x^\beta = \hat x + \beta \hat d$. It is easy to verify that $x^\beta \in F$ for each $\beta \in [\alpha_l - \alpha,\alpha_u - \alpha]$. We claim that $\ell_R(x^\beta)$ is a linear function of $\beta$ on this interval. By \eqref{ell_R_closed}, it suffices to show that each term is a linear function. 

First, let us focus on the term given by $
\max\{0, x^\beta_i + x^\beta_j - 1\} = 
\max\{0, \hat{x}_i + \hat{x}_j - 1 + \beta \hat d_i + \beta \hat d_j\}$, where $i = 1,\ldots,n;~j = 1,\ldots,n$. It suffices to show that the sign of $\hat{x}_i + \hat{x}_j - 1 + \beta \hat d_i + \beta \hat d_j$ does not change for each $\beta \in [\alpha_l - \alpha,\alpha_u - \alpha]$ and for each $i = 1,\ldots,n;~j = 1,\ldots,n$. Clearly, $\hat{x}_i + \hat{x}_j - 1 + \beta \hat d_i + \beta \hat d_j = \hat{x}_i + \hat{x}_j - 1$ if $\{i,j\} \subseteq \mathbb{I}_3$; or $i \in \mathbb{I}_1, j \in \mathbb{I}_2$; or $i \in \mathbb{I}_2, j \in \mathbb{I}_1$. For the remaining cases, it follows from the definitions of $\mathbb{I}_1, \mathbb{I}_2$, and $\mathbb{I}_3$ that
\begin{footnotesize}
\[
\hat{x}_i + \hat{x}_j - 1 + \beta \hat d_i + \beta \hat d_j \in \left\{ \begin{array}{ll} [2 \alpha_l - 1, 2 \alpha_u - 1], & \{i,j\} \subseteq \mathbb{I}_1, \\ 
 {[1 - 2 \alpha_u, 1 - 2 \alpha_l]}, 
 & \{i,j\} \subseteq \mathbb{I}_2, \\
 {[\alpha_l - 1,\alpha_l + \alpha_u - 1] \cup [\alpha_l + \alpha_u - 1,0] \cup [0,\alpha_u]}, & i \in \mathbb{I}_1, j \in \mathbb{I}_3; \textrm{ or }i \in \mathbb{I}_3, j \in \mathbb{I}_1,\\
{[-\alpha_u, 0] \cup [0, 1 - \alpha_l - \alpha_u] \cup [1 - \alpha_l - \alpha_u, 1 - \alpha_l]}, & i \in \mathbb{I}_2, j \in \mathbb{I}_3; \textrm{ or }i \in \mathbb{I}_3, j \in \mathbb{I}_2.
\end{array} \right. 
\]
\end{footnotesize}
Our claim now follows from $0 \leq \alpha_l < \alpha_u \leq \textstyle \frac{1}{2}$. Therefore, $\max\{0, \hat{x}_i + \hat{x}_j - 1 + \beta \hat d_i + \beta \hat d_j\}$ is a linear function on $\beta \in [\alpha_l - \alpha,\alpha_u - \alpha]$ for each $i = 1,\ldots,n;~ j = 1,\ldots,n$.

Let us now consider the term $\min\{x^\beta_i,x^\beta_j\} = \min\{\hat{x}_i + \beta \hat d_i,\hat{x}_j + \beta \hat d_j\}$. By the choice of $\hat d$, it is easy to see that the order of the components of $x^\beta$ remains unchanged for each $\beta \in [\alpha_l - \alpha,\alpha_u - \alpha]$, i.e., if $\hat{x}_i \leq \hat x_j$, then $\hat{x}_i + \beta \hat d_i \leq \hat{x}_j + \beta \hat d_j$ for each $i = 1,\ldots,n;~ j = 1,\ldots,n$. It follows that $\min\{\hat{x}_i + \beta \hat d_i,\hat{x}_j + \beta \hat d_j\}$ is a linear function on $\beta \in [\alpha_l - \alpha,\alpha_u - \alpha]$. 

Since the third term in \eqref{ell_R_closed} is also a linear function on $\beta \in [\alpha_l - \alpha,\alpha_u - \alpha]$, it follows that $\ell_R(x^\beta)$ is a linear function on $[\alpha_l - \alpha,\alpha_u - \alpha]$. Therefore, by the optimality of $\hat x$ in (R1), $\ell_R(x^\beta)$ is a constant function on this interval. If $\alpha_l = 0$ and $\alpha_u = \textstyle \frac{1}{2}$, then the alternate optimal solution at each end of this interval satisfies the desired property. Otherwise, by moving to the solution at one of the end points with $\alpha_l > 0$ or $\alpha_u < \textstyle \frac{1}{2}$, one can repeat the same procedure in an iterative manner to arrive at an alternate optimal solution with the desired property. Note that this procedure is finite since either $\alpha_l$ strictly decreases or $\alpha_u$ strictly increases at each iteration. 
\end{proof}

We can utilize Proposition~\ref{half_frac} to obtain the following result about the set of optimal solutions of (R).

\begin{corollary} \label{opt_sol_RLT}
There exists an optimal solution $(\hat x,\hat X) \in \R^n \times \cS^n$ of the RLT relaxation (R) such that $\hat x_j \in \{0,\textstyle \frac{1}{2},1\}$ for each $j = 1,\ldots,n$ and $\hat X_{ij} \in \{0,\textstyle \frac{1}{2},1\}$ for each $i = 1,\ldots,n;~j = 1,\ldots,n$ such that 
\[
\begin{bmatrix} \hat X_{\mathbb{L}\mathbb{L}} & \hat X_{\mathbb{L}\mathbb{B}} & \hat X_{\mathbb{L}\mathbb{U}} \\ \hat X_{\mathbb{B}\mathbb{L}} & \hat X_{\mathbb{B}\mathbb{B}} & \hat X_{\mathbb{B}\mathbb{U}} \\
\hat X_{\mathbb{U}\mathbb{L}} & \hat X_{\mathbb{U}\mathbb{B}} & \hat X_{\mathbb{U}\mathbb{U}} \end{bmatrix} = \begin{bmatrix} 0 & 0 & 0 \\ 0 & \hat X_{\mathbb{B}\mathbb{B}} & \frac{1}{2} e_{\mathbb{B}} e_{\mathbb{U}}^T \\
0 & \frac{1}{2} e_{\mathbb{U}} e_{\mathbb{B}}^T & e_{\mathbb{U}} e_{\mathbb{U}}^T \end{bmatrix}, \quad \hat X_{ij} \in \{0,\textstyle \frac{1}{2}\}, \quad i \in \mathbb{B},~j \in \mathbb{B}.
\]
\end{corollary}
\begin{proof}
By Proposition~\ref{half_frac}, there exists $\hat x \in F$ such that $\ell^*_R =\ell_R(\hat x)$ and $\hat x_j \in \{0,\textstyle \frac{1}{2},1\}$ for each $j = 1,\ldots,n$. Define $\hat X \in \cS^n$ such that 
\[
\hat X_{ij} = \left\{\begin{array}{ll} \max\{0,\hat{x}_i + \hat{x}_j - 1\}, & \textrm{if}~Q_{ij} > 0, \\
\min\{\hat{x}_i,\hat{x}_j\}, & \textrm{if}~Q_{ij} < 0, \\
0, & \textrm{otherwise.} \end{array} \right.
\]
Note that $(\hat x,\hat X) \in \cF_R$ by \eqref{def_cFR_alt} and $\ell_R(\hat x) = \textstyle \frac{1}{2}\langle Q, \hat X \rangle + c^T \hat x$ by Lemma~\ref{piecewise-lin}. Therefore, $(\hat x,\hat X)$ is an optimal solution of (R) with the desired property.
\end{proof}    

The next result follows from Corollary~\ref{opt_sol_RLT}.

\begin{corollary} \label{vertices_cFR}
For each vertex $(\hat x,\hat X) \in \cF_R$, $\hat x_j \in \{0,\textstyle \frac{1}{2},1\}$ for each $j = 1,\ldots,n$ and $\hat X_{ij} \in \{0,\textstyle \frac{1}{2},1\}$ for each $i = 1,\ldots,n;~j = 1,\ldots,n$.
\end{corollary}
\begin{proof}
Since $\cF_R$ is a polytope, $(\hat x,\hat X) \in \cF_R$ is a vertex if and only if there exists a $(Q,c) \in \cS^n \times \R^n$ such that $(\hat x,\hat X) \in \cF_R$ is the unique optimal solution of (R). The assertion follows from Corollary~\ref{opt_sol_RLT}.  
\end{proof}

We remark that Padberg~\cite{padberg1989boolean} established a similar result for the set $\cF_R^-$ given by \eqref{def_cFR-}. Corollary~\ref{vertices_cFR} extends the same result to $\cF_R$. In contrast with the proof of~\cite{padberg1989boolean}, which relies on linearly independent active constraints, our proof uses a specific property of the set of optimal solutions of the reduced formulation (R1). 

\subsection{First Description of Exact RLT Relaxations}

In this section, we present our first description of the set of instances of (BoxQP) with an exact RLT relaxation. We start with a useful property of such instances.

\begin{proposition} \label{exact-rlt}
For any instance of (BoxQP), the RLT relaxation is exact, i.e., $\ell^*_R = \ell^*$, if and only if there exists a vertex $v \in F$ such that $v$ is an optimal solution of (R1), where (R1) is given by \eqref{def_R1}.
\end{proposition}
\begin{proof}
Suppose that $\ell^*_R = \ell^*$. Then, by \eqref{conv_underest_rel_1} and \eqref{def_R1}, for any optimal solution $x^* \in F$ of (BoxQP), we have $q(x^*) = \ell^* = \ell^*_R \leq \ell_R(x^*) \leq q(x^*) = \ell^*$, which implies that $\ell^*_R = \ell_R(x^*) = q(x^*)$.
By Lemma~\ref{exact_underests}(i), either $x^*$ is a vertex of $F$, in which case, we are done, or $Q_{\mathbb{B}\mathbb{B}} = 0$, where $\mathbb{B}$ is given by \eqref{def_indB}. In the latter case, since $\hat x^*_{\mathbb{L}} = 0$ and $\hat x^*_{\mathbb{U}} = e_\mathbb{U}$ by \eqref{def_indL} and \eqref{def_indU}, respectively, 
we obtain
\[
\ell^*_R = \ell^* = q(x^*) = \frac{1}{2} e_\mathbb{U}^T Q_{\mathbb{U}\mathbb{U}} e_\mathbb{U} + (x^*_{\mathbb{B}})^T Q_{\mathbb{B}\mathbb{U}} e_\mathbb{U} + c_{\mathbb{U}}^T e_\mathbb{U} + c_{\mathbb{B}}^T x^*_{\mathbb{B}} = \frac{1}{2} e_\mathbb{U}^T Q_{\mathbb{U}\mathbb{U}} e_\mathbb{U} + c_{\mathbb{U}}^T e_\mathbb{U},
\]
where the last equality follows from the identity $Q_{\mathbb{B}\mathbb{U}} e_\mathbb{U} + c_{\mathbb{B}} = 0$ by \eqref{kkt1}--\eqref{kkt3}. Therefore, for the vertex $v \in F$ given by $v_j = 1$ for each $j \in \mathbb{U}$ and $v_j = 0$ for each $j \in \mathbb{L} \cup \mathbb{B} = 0$, we obtain $q(v) = \ell^* = \ell_R(v)$ by Lemma~\ref{exact_underests}(i).

Conversely, if there exists a vertex $v \in F$ such that $\ell_R(v) = \ell^*_R$, then we have $\ell^* \leq q(v) = \ell_R(v) = \ell^*_R$ by Lemma~\ref{exact_underests}(i). The assertion follows from \eqref{convex_rels}.   
\end{proof}

Proposition~\ref{exact-rlt} presents an important property of the set of instances of (BoxQP) with exact RLT relaxations in terms of the set of optimal solutions and gives rise to the following corollary.

\begin{corollary} \label{vertex-optimality}
For any instance of (BoxQP), the RLT relaxation is exact if and only if there exists a vertex $v \in F$ such that $(v,vv^T)$ is an optimal solution of (R). Furthermore, in this case, $v$ is an optimal solution of (BoxQP).   
\end{corollary}
\begin{proof}
The assertion directly follows from Proposition~\ref{exact-rlt}, Lemma~\ref{exact_underests}(i), and Corollary~\ref{vertex-prop}.    
\end{proof}

By Corollary~\ref{vertex-optimality}, if the set of optimal solutions of (BoxQP) does not contain a vertex, then the RLT relaxation is inexact. Note that if $q(\cdot)$ is a concave function, then the set of optimal solutions of (BoxQP) contains at least one vertex. However, this is an NP-hard problem (see, e.g.,~\cite{Sahni74}), which implies that the RLT relaxation can be inexact even if the set of optimal solutions of (BoxQP) contains at least one vertex. The next example illustrates that the RLT relaxation can be inexact even if every optimal solution of (BoxQP) is a vertex.

\begin{example} \label{Ex_1}
Consider an instance of (BoxQP) with
\[
Q = \begin{bmatrix}
-1&-2\\-2&1
\end{bmatrix}, \quad c = \begin{bmatrix}
1 \\ 1
\end{bmatrix}.
\]
Note that
\[
q(x) = \frac{1}{2}\left(-x_1^2 + x_2^2 - 4 x_1 x_2\right) + x_1 + x_2 = \frac{1}{2}\left(x_2 - x_1\right)^2 + (x_1 + x_2) (1 - x_1),
\]
which implies that $q(x) \geq 0$ for each $x \in F$ and $q(x) = 0$ if and only if $x \in \{0,e\}$. Therefore, $\ell^* = 0$ and the set of optimal solution is given by $\{0,e\}$, which consists of two vertices of $F$. However, for $\hat x = \textstyle \frac{1}{2}e \in \R^2$, it is easy to verify that $\ell_R(\hat x) = -1/4 < \ell^*$. In fact, $\ell^*_R = \ell_R(\hat x) = -1/4$ and $\ell_R(x) > \ell_R(\hat x)$ for each $x \in F \backslash \{\hat x\}$. Figure~\ref{inexact_Ex1} illustrates the functions $q(\cdot)$ and $\ell_R(\cdot)$. 
\end{example}

\begin{figure}[hbt]
\label{inexact_Ex1}
\includegraphics[scale=0.4]{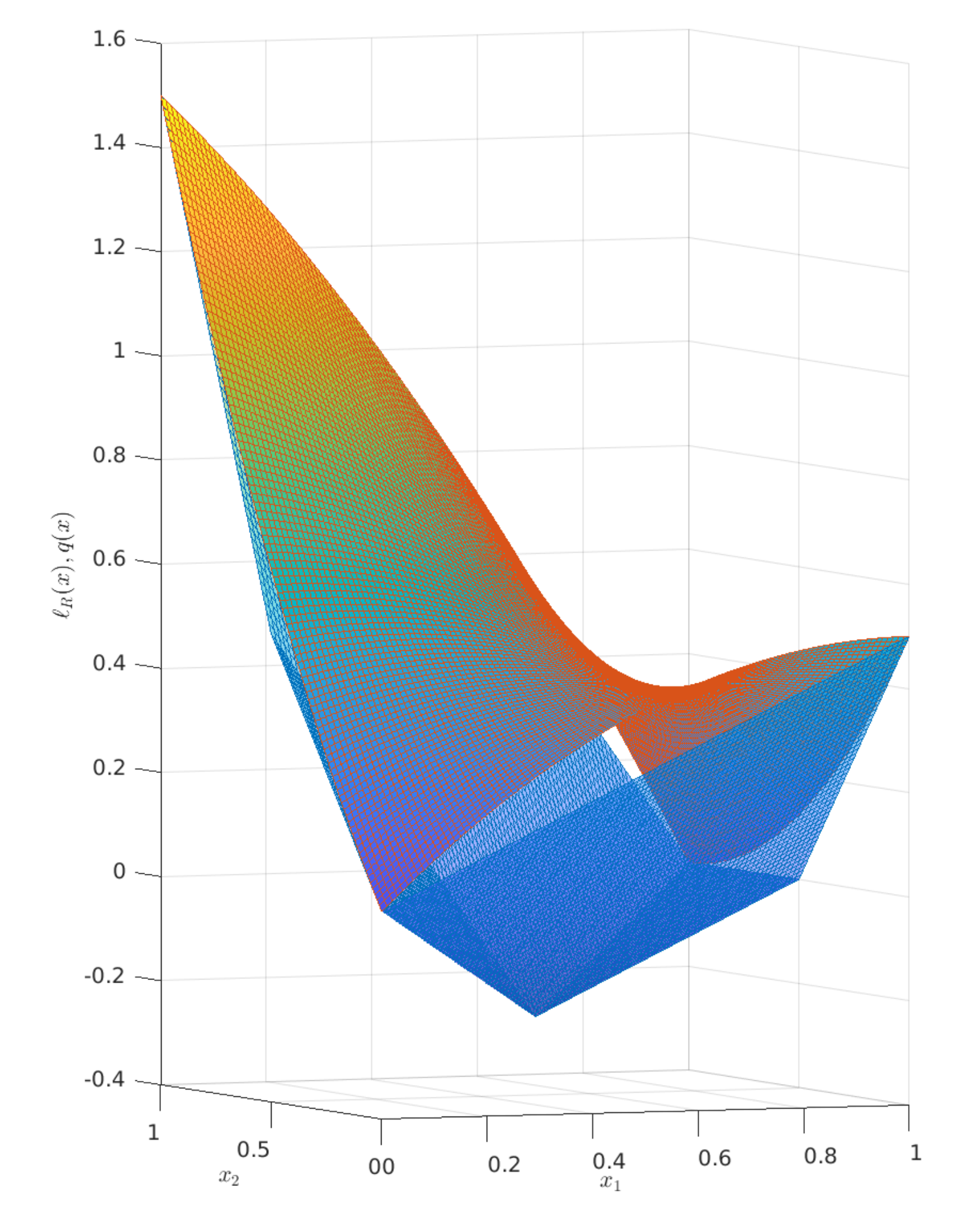}
\centering
\caption{Graphs of $q(\cdot)$ (above) and $\ell_R(\cdot)$ (below) for Example~\ref{Ex_1}}
\end{figure}

We next present our first description of the set of instances that admit an exact RLT relaxation. 
To that end, let us define
\begin{equation} \label{exact_RLT_set}
\cE_R = \{(Q,c) \in \cS^n \times \R^n: \ell^* = \ell^*_R\}.    
\end{equation}
i.e., $\cE_R$ denotes the set of all instances of (BoxQP) that admit an exact RLT relaxation. 
For a given $\hat x \in F$, let us define
\begin{equation} \label{def_cER_x}
 \cO_R(\hat x) = \{(Q,c) \in \cS^n \times \R^n: \ell_R(\hat x) = \ell^*_R\} = \{(Q,c) \in \cS^n \times \R^n: \ell_R(\hat x) \leq \ell_R(x), \quad x \in F\},   
\end{equation}
i.e., $\cO_R(\hat x)$ denotes the set of instances of (BoxQP) such that $\hat x$ is a minimizer of (R1), where (R1) is given by \eqref{def_R1}. 
By \eqref{ell_R_closed}, it is easy to see that $\cO_R(\hat x)$ is a convex cone for each $\hat x \in F$. Our next result provides an algebraic description of the set $\cE_R$.

\begin{proposition} \label{exact_RLT_char1}
Let $V \subset F$ denote the set of vertices of $F$ given by \eqref{vertices of box QP} and let $V^+ = \{x \in F: x_j \in \{0,\textstyle \frac{1}{2},1\}, \quad j = 1,\ldots,n\}$. Then, $\cE_R$ defined as in \eqref{exact_RLT_set} is given by the union of a finite number of polyhedral cones and admits the following description:
\begin{equation} \label{exact_RLT_desc1}
 \cE_R = \bigcup_{v \in V} \cO_R(v) = \bigcup_{v \in V} \left(\bigcap_{\hat x \in V^+} \left\{(Q,c) \in \cS^n \times \R^n: \ell_R(v) \leq \ell_R(\hat x)\right\} \right).  
\end{equation}
\end{proposition}
\begin{proof}
By Proposition~\ref{exact-rlt}, the RLT relaxation is exact if and only if there exists a vertex $v \in F$ such that $\ell_R(v) = \ell^*_R$,  
which, together with \eqref{def_cER_x}, leads to the first equality in \eqref{exact_RLT_desc1}. By Proposition~\ref{half_frac}, the set $V_+$ contains at least one minimizer of $\ell_R(\cdot)$, 
which implies the second equality in \eqref{exact_RLT_desc1}. $\cE_R$ is the union of a finite number of polyhedral cones since $\ell_R(x)$ is a linear function of $(Q,c)$ for each fixed $x \in F$ by Lemma~\ref{piecewise-lin} and $V^+$ is a finite set.
\end{proof}

For each $v \in V$, we remark that Proposition~\ref{exact_RLT_char1} gives a description of $\cO_R(v)$ using $3^n$ linear inequalities since $|V^+| = 3^n$. In fact, for each $v \in V$, the convexity of $\ell_R(\cdot)$ on $F$ implies that it suffices to consider only those $\hat x \in V_+$ such that $\hat x_j \in \{0,\textstyle \frac{1}{2}\}$ for $j \in \mathbb{L}(v)$ and $\hat x_j \in \{\textstyle \frac{1}{2},1\}$ for $j \in \mathbb{U}(v)$, where $\mathbb{L}(v)$ and $\mathbb{U}(v)$ are given by \eqref{def_indL} and \eqref{def_indU}, respectively, which implies a simpler description  of $\cE_R$ with $2^n$ inequalities. Due to the exponential number of such inequalities, this description is not very useful for efficiently checking if a particular instance of (BoxQP) admits an exact RLT relaxation. Similarly, this description cannot be used easily for constructing such an instance of (BoxQP). In the next section, we present an alternative description of $\cE_R$ using linear programming duality, which gives rise to algorithms for efficiently constructing instances of (BoxQP) with exact or inexact RLT relaxations.

\subsection{An Alternative Description of Exact RLT Relaxations} \label{another_char_RLT}

In this section, our main goal is to present an alternative description of the set $\cE_R$ using duality. 

Recall that the RLT relaxation is given by
\[
\textrm{(R)} \quad \ell^*_R = \min\limits_{(x,X) \in \R^n \times \cS^n} \left\{\frac{1}{2}\langle Q, X\rangle + c^T x: (x,X) \in \cF_R\right\},
\]
where, $\cF_R$ given by \eqref{def_cFR_alt}, can be expressed in the following form:
\begin{equation} \label{def_cFR}
\cF_R = \left\{(x,X) \in \R^n \times \cS^n: \begin{array}{rcl} x & \leq & e\\x & \geq & 0\\X - xe^T - ex^T + ee^T & \geq & 0\\
-X + ex^T & \geq & 0\\
X & \geq & 0
\end{array}
\right\}.
\end{equation}

By defining dual variables $(u,v,W,Y,Z) \in \R^n \times \R^n \times \cS^n \times \R^{n \times n} \times \cS^n$ corresponding to each of the five constraints in \eqref{def_cFR}, respectively, the dual problem of (R) is given by
\[
\begin{array}{llrcl}
\textrm{(R-D)} & \max\limits_{(u,v,W,Y,Z) \in \R^n \times \R^n \times \cS^n \times \R^{n \times n} \times \cS^n} & -e^T u - \frac{1}{2} e^T W e & & \\
 & \textrm{s.t.} & & & \\
 & & -u + v - W e + Y^T e & = & c \\
 & & W - Y - Y^T + Z & = & Q \\
 & & u & \geq & 0 \\
 & & v & \geq & 0 \\
 & & W & \geq & 0 \\
 & & Y & \geq & 0 \\
 & & Z & \geq & 0.
\end{array}
\]

Note that the variables $(W,Y,Z) \in \cS^n \times \R^{n \times n} \times \cS^n$ in (R-D) are scaled by a factor of $\textstyle \frac{1}{2}$. First, we start with optimality conditions for (R) and (R-D). 

\begin{lemma} \label{opt-cond-R-RD}
$(\hat x, \hat X) \in \cF_R$ is an optimal solution of (R) if and only if there exists $(\hat u,\hat v,\hat W,\hat Y,\hat Z) \in \R^n \times \R^n \times \cS^n \times \R^{n \times n} \times \cS^n$ such that
\begin{eqnarray}
Q & = & \hat W - \hat Y - \hat Y^T + \hat Z \label{RLTcond00}\\
c & = & -\hat u + \hat v - \hat W e + \hat Y^T e \label{RLTcond01} \\
 \hat u^T (e - \hat x) & = & 0 \label{RLTcond1}\\
 \hat v^T \hat x & =& 0 \label{RLTcond2}\\
 \langle \hat W, \hat X - \hat x e^T - e \hat x^T + ee^T \rangle & = & 0 \label{RLTcond3}\\
 \langle \hat Y, e \hat x^T - \hat X \rangle & = & 0 \label{RLTcond4}\\
 \langle \hat Z, \hat X \rangle & = & 0 \label{RLTcond5} \\
 \hat u & \geq & 0 \label{RLTcond6} \\
\hat v & \geq & 0 \label{RLTcond7} \\
\hat W & \geq & 0 \label{RLTcond8} \\
\hat Y & \geq & 0 \label{RLTcond9} \\
\hat Z & \geq & 0. \label{RLTcond10}
\end{eqnarray}
\end{lemma}
\begin{proof}
The assertion follows from strong duality since each of (R) and (R-D) is a linear programming problem.
\end{proof}

Lemma~\ref{opt-cond-R-RD} gives rise to an alternative description of the set of instances of (BoxQP) that admit an exact RLT relaxation.

\begin{corollary} \label{alt_char_exact_RLT}
$(Q,c) \in \cE_R$, where $\cE_R$ is defined as in \eqref{exact_RLT_set}, if and only if there exists a vertex $v \in F$ and there exists $(\hat u,\hat v,\hat W,\hat Y,\hat Z) \in \R^n \times \R^n \times \cS^n \times \R^{n \times n} \times \cS^n$ such that the relations \eqref{RLTcond00}--\eqref{RLTcond10} hold, where $(\hat x,\hat X) = (v,vv^T)$.    
\end{corollary}
\begin{proof}
Note that $(Q,c) \in \cE_R$ if and only if there exists a vertex $v \in F$ such that $(v,vv^T)$ is an optimal solution of (R) by Corollary~\ref{vertex-optimality}. The assertion now follows from Lemma~\ref{opt-cond-R-RD}.   
\end{proof}

In the next section, we discuss how Corollary~\ref{alt_char_exact_RLT} can be utilized to construct instances of (BoxQP) with exact and inexact RLT relaxations. 

\subsection{Construction of Instances with Exact RLT Relaxations}

In this section, we describe an algorithm for constructing instances of (BoxQP) with an exact RLT relaxation. Algorithm~\ref{Alg1} is based on designating a vertex $v \in F$ and constructing an appropriate dual feasible solution that satisfies optimality conditions together with $(v,vv^T) \in \cF_R$.

\begin{algorithm}
\begin{algorithmic}[1]
\Require $n; \mathbb{L} \subseteq \{1,\ldots,n\}$
\Ensure $(Q,c) \in \cE_R$
\State $\mathbb{U} \gets \{1,\ldots,n\} \backslash \mathbb{L}$
\State Choose an arbitrary $\hat u \in \R^n$ such that $\hat u \geq 0$ and $\hat u_\mathbb{L} = 0$.
\State Choose an arbitrary $\hat v \in \R^n$ such that $\hat v \geq 0$ and $\hat v_\mathbb{U} = 0$.
\State Choose an arbitrary $\hat W \in \cS^n$ such that $\hat W \geq 0$ and $\hat W_{\mathbb{L}\mathbb{L}} = 0$.
\State Choose an arbitrary $\hat Y \in \R^{n \times n}$ such that $\hat Y \geq 0$ and $\hat Y_{\mathbb{L}\mathbb{U}} = 0$.
\State Choose an arbitrary $\hat Z \in \cS^n$ such that $\hat Z \geq 0$ and $\hat Z_{\mathbb{U}\mathbb{U}} = 0$.
\State $Q \gets \hat W - \hat Y - \hat Y^T + \hat Z$
\State $c \gets -\hat u + \hat v - \hat W e + \hat Y^T e$
\end{algorithmic}
\caption{(BoxQP) Instance with an Exact RLT Relaxation}
\label{Alg1}
\end{algorithm}

The following result establishes the correctness of Algorithm~\ref{Alg1}.

\begin{proposition} \label{alg1-correct}
Algorithm~\ref{Alg1} returns $(Q,c) \in \cE_R$, where $\cE_R$ is defined as in \eqref{exact_RLT_set}. Conversely, any $(Q,c) \in \cE_R$ can be generated by Algorithm~\ref{Alg1} with appropriate choices of $\mathbb{L} \subseteq \{1,\ldots,n\}$ and $(\hat u,\hat v,\hat W,\hat Y,\hat Z) \in \R^n \times \R^n \times \cS^n \times \R^{n \times n} \times \cS^n$.  
\end{proposition}
\begin{proof}
Let $\mathbb{L} \subseteq \{1,\ldots,n\}$ and define $\mathbb{U} = \{1,\ldots,n\} \backslash \mathbb{L}$. Let $v \in F$ be the vertex given by $v_j = 0,~j \in \mathbb{L}$ and $v_j = 1,~j \in \mathbb{U}$. It is easy to verify that $(\hat u,\hat v,\hat W,\hat Y,\hat Z) \in \R^n \times \R^n \times \cS^n \times \R^{n \times n} \times \cS^n$ and $(\hat x,\hat X) = (v,vv^T)$ satisfy the hypotheses of Corollary~\ref{alt_char_exact_RLT}, which establishes the first assertion. The second assertion also follows from Corollary~\ref{alt_char_exact_RLT}.
\end{proof}

By considering all possible subsets $\mathbb{L} \subseteq \{1,\ldots,n\}$, Proposition~\ref{alg1-correct} yields an alternative characterization of $\cE_R$ given by the union of $2^n$ polyhedral cones (cf.~Proposition~\ref{exact_RLT_char1}). In contrast with the description in Proposition~\ref{exact_RLT_char1}, this alternative description enables us to easily construct an instance of (BoxQP) with a known optimal vertex and an exact RLT relaxation (cf.~Corollary~\ref{vertex-optimality}). Note, however, that even the alternative description is not very useful for effectively checking if $(Q,c) \in \cE_R$ due to the exponential number of polyhedral cones.

\subsection{Construction of Instances with Inexact RLT Relaxations} \label{const_inexact_rlt}

In this section, we propose an algorithm for constructing instances of (BoxQP) with an inexact RLT relaxation. Algorithm~\ref{Alg2} is based on constructing a dual optimal solution of (R-D) such that no feasible solution of the form $(v, vv^T) \in \cF_R$ can be an optimal solution of (R), where $v \in F$ is a vertex.

\begin{algorithm}
\begin{algorithmic}[1]
\Require $n; \mathbb{B} \subseteq \{1,\ldots,n\}; \mathbb{B} \neq \emptyset$
\Ensure $(Q,c) \not \in \cE_R$
\State Choose an arbitrary $\mathbb{L} \subseteq \{1,\ldots,n\} \backslash \mathbb{B}$
\State $\mathbb{U} \gets \{1,\ldots,n\} \backslash (\mathbb{B} \cup \mathbb{L})$
\State Choose an arbitrary $k \in \mathbb{B}$.
\State Choose an arbitrary $\hat u \in \R^n$ such that $\hat u_\mathbb{U} \geq  0$ and $\hat u_{\mathbb{L} \cup \mathbb{B}} = 0$.
\State Choose an arbitrary $\hat v \in \R^n$ such that $\hat v_\mathbb{L} \geq 0$ and $\hat v_{\mathbb{B} \cup \mathbb{U}} = 0$.
\State Choose an arbitrary $\hat W \in \cS^n$ such that $\hat W_{\mathbb{L}\mathbb{L}} = 0,~\hat W_{\mathbb{L}\mathbb{B}} = 0,~\hat W_{\mathbb{B}\mathbb{L}} = 0,~\hat W_{kk} > 0$, and $\hat W_{ij} \geq 0$ otherwise.
\State Choose an arbitrary $\hat Y \in \R^{n \times n}$ such that $\hat Y_{\mathbb{L}\mathbb{B}} = 0,~\hat Y_{\mathbb{L}\mathbb{U}} = 0,~\hat Y_{\mathbb{B}\mathbb{B}} = 0,~\hat Y_{\mathbb{B}\mathbb{U}} = 0$ and $\hat Y_{ij} \geq 0$ otherwise.
\State Choose an arbitrary $\hat Z \in \cS^n$ such that $\hat Z_{\mathbb{B}\mathbb{U}} = 0,~\hat Z_{\mathbb{U}\mathbb{B}} = 0,~\hat Z_{\mathbb{U}\mathbb{U}} = 0,~\hat Z_{kk} > 0$,
and $\hat Z_{ij} \geq 0$ otherwise.
\State $Q \gets \hat W - \hat Y - \hat Y^T + \hat Z$
\State $c \gets -\hat u + \hat v - \hat W e + \hat Y^T e$
\end{algorithmic}
\caption{(BoxQP) Instance with an Inexact RLT Relaxation}
\label{Alg2}
\end{algorithm}

The next result establishes that the output from Algorithm~\ref{Alg2} is an instance of (BoxQP) with an inexact RLT relaxation.

\begin{proposition} \label{Alg2correct}
Algorithm~\ref{Alg2} returns $(Q,c) \not \in \cE_R$, where $\cE_R$ is defined as in \eqref{exact_RLT_set}.    
\end{proposition}
\begin{proof}
Consider the partition $(\mathbb{L},\mathbb{B},\mathbb{U})$ of the index set $\{1,\ldots,n\}$ as defined in Algorithm~\ref{Alg2}, where $\mathbb{B} \neq \emptyset$. Clearly, $(\hat u,\hat v,\hat W,\hat Y,\hat Z) \in \R^n \times \R^n \times \cS^n \times \R^{n \times n} \times \cS^n$ is a feasible solution of (R-D). We will construct a feasible solution $(\hat x, \hat X) \in \R^n \times \cS^n$ of (R) that satisfies the optimality conditions of Lemma~\ref{opt-cond-R-RD}. 

Consider the following solution $(\hat x, \hat X) \in \R^n \times \cS^n$:
\[
\hat x_{\mathbb{L}} = 0, \quad \hat x_{\mathbb{B}} = \textstyle \frac{1}{2} e_\mathbb{B} \quad \hat x_{\mathbb{U}} = e_\mathbb{U},
\]
and
\[
\begin{bmatrix} \hat X_{\mathbb{L}\mathbb{L}} & \hat X_{\mathbb{L}\mathbb{B}} & \hat X_{\mathbb{L}\mathbb{U}} \\ \hat X_{\mathbb{B}\mathbb{L}} & \hat X_{\mathbb{B}\mathbb{B}} & \hat X_{\mathbb{B}\mathbb{U}} \\
\hat X_{\mathbb{U}\mathbb{L}} & \hat X_{\mathbb{U}\mathbb{B}} & \hat X_{\mathbb{U}\mathbb{U}} \end{bmatrix} = \begin{bmatrix} 0 & 0 & 0 \\ 0 & 0 & \textstyle \frac{1}{2} e_\mathbb{B} e_{\mathbb{U}}^T \\
0 & \textstyle \frac{1}{2} e_{\mathbb{U}} e_{\mathbb{B}}^T & e_{\mathbb{U}} e_{\mathbb{U}}^T \end{bmatrix}.
\]
By Lemma~\ref{feas_cond}, $(\hat x, \hat X) \in \cF_R$. By Steps 4, 5, 6, 7, and 8 of Algorithm~\ref{Alg2}, it is easy to verify that \eqref{RLTcond1}, \eqref{RLTcond2}, \eqref{RLTcond3}, \eqref{RLTcond4}, and \eqref{RLTcond5} are respectively satisfied. Therefore, by Lemma~\ref{opt-cond-R-RD}, we conclude that $(\hat x, \hat X)$ is an optimal solution of (R) and $(\hat u,\hat v,\hat W,\hat Y,\hat Z)$ is an optimal solution of (R-D).

We next argue that the RLT relaxation is inexact. Let $(\tilde x, \tilde X) \in \cF_R$ be an arbitrary optimal solution of (R). By Lemma~\ref{feas_cond}, $(\tilde x, \tilde X)$ and $(\hat u,\hat v,\hat W,\hat Y,\hat Z)$ satisfy the conditions \eqref{RLTcond1}, \eqref{RLTcond2}, \eqref{RLTcond3}, \eqref{RLTcond4}, and \eqref{RLTcond5}. 
By \eqref{RLTcond5} and Step 8 of Algorithm~\ref{Alg2}, we obtain $\tilde X_{kk} = 0$ since $\hat Z_{kk} > 0$.  
Since $\hat W_{kk} > 0$ by Step 6 of Algorithm~\ref{Alg2}, the relation \eqref{RLTcond3} implies that $\tilde X_{kk} - 2 \tilde x_k + 1 = 0$, i.e., $\tilde x_k = \textstyle \frac{1}{2}$ since $\tilde X_{kk} = 0$. By Lemma~\ref{opt-cond-R-RD}, we conclude that $\tilde x_k = \textstyle \frac{1}{2}$ for each optimal solution $(\tilde x, \tilde X) \in \cF_R$ of (R). By Corollary~\ref{vertex-optimality}, we conclude that $(Q,c) \not \in \cE_R$.
\end{proof}

Algorithm~\ref{Alg2} can be used to generate an instance of (BoxQP) with an inexact RLT relaxation. 
Note that Algorithm~\ref{Alg2} constructs an instance $(Q,c)$ with the property that at least one component $\hat x_k$ is fractional at every optimal solution $(\hat x, \hat X)$ of (R), which is sufficient for having an inexact RLT relaxation by Corollary~\ref{vertex-optimality}. However, this condition may not be necessary for an inexact RLT relaxation. In particular, note that an instance generated by Algorithm~\ref{Alg2} cannot have a concave objective function since $Q_{kk} = \hat W_{kk} + \hat Z_{kk} > 0$. On the other hand, for the specific instance $(Q,c) \in \cS^3 \times \R^3$ in~\cite{anstreicher2010computable} given by $Q = \frac{1}{3}ee^T - I$, where $I \in \cS^3$ denotes the identity matrix, and $c = 0$, the objective function is concave and the optimal value is given by $\ell^* = -\textstyle\frac{1}{3}$, which is attained at any vertex that has exactly one component equal to 1. For $\hat x = \textstyle\frac{1}{2}e \in \R^3$, we have $\ell_R(\hat x) = -\frac{1}{2} < -\textstyle\frac{1}{3} = \ell^*$ by Lemma~\ref{piecewise-lin}, which implies that the RLT relaxation is inexact on this instance. Therefore, in contrast with Algorithm~\ref{Alg1}, we conclude that Algorithm~\ref{Alg2} may not necessarily generate all possible instances $(Q,c) \not \in \cE_R$.

\section{Exact and Inexact SDP-RLT Relaxations} \label{Sec5}

In this section, we focus on the set of instances of (BoxQP) that admit exact and inexact SDP-RLT relaxations. We give a complete algebraic description of the set of instances of (BoxQP) that admit an exact SDP-RLT relaxation. In addition, we develop an algorithm for constructing such an instance of (BoxQP) as well as for constructing an instance of (BoxQP) with an exact SDP-RLT relaxation but an inexact RLT relaxation. 

Similar to the RLT relaxation, let us define
\begin{equation} \label{exact_RLT_SDP_set}
\cE_{RS} = \{(Q,c) \in \cS^n \times \R^n: \ell^* = \ell_{RS}\},    
\end{equation}
i.e., $\cE_{RS}$ denotes the set of all instances of (BoxQP) that admit an exact SDP-RLT relaxation. 
By \eqref{convex_rels}, the SDP-RLT relaxation of any instance of (BoxQP) is at least as tight as the RLT relaxation. It follows that 
\begin{equation} \label{cERsubsCERS}
\cE_R \subseteq \cE_{RS}, 
\end{equation}
where $\cE_R$ is given by \eqref{exact_RLT_set}. 

By Corollary~\ref{implications} and \eqref{def_RS1}, we clearly have $(Q,c) \in \cE_{RS}$ whenever $Q \succeq 0$. Furthermore, the SDP-RLT relaxation is always exact (i.e., $\cE_{RS} = \cS^n \times \R^n$) if and only if $n \leq 2$~\cite{anstreicher2010computable}.

For the RLT relaxation, Proposition~\ref{half_frac} established the existence of an optimal solution of the RLT relaxation with a particularly simple structure. This observation enabled us to characterize the set of instances of (BoxQP) with an exact RLT relaxation as the union of a finite number polyhedral cones (see Proposition~\ref{exact_RLT_char1}). In contrast, the next result shows that the set of optimal solutions of the SDP-RLT relaxation cannot have such a simple structure.

\begin{lemma} \label{SDP-RLT-opt-char}
 For any $\hat x \in F$, there exists an instance $(Q,c)$ of (BoxQP) such that $\hat x$ is the unique optimal solution of (RS1), where (RS1) is given by \eqref{def_RS1}, and $(Q,c) \in \cE_{RS}$. 
\end{lemma}
\begin{proof}
 For any $\hat x \in F$, consider an instance of (BoxQP) with $(Q,c) \in \cS^n \times \R^n$, where $Q \succ 0$ and $c = -Q \hat x$. We obtain $q(x) = \textstyle\frac{1}{2}\left((x - \hat x)^T Q (x - \hat x) - \hat x^T Q \hat x\right)$. Since $Q \succ 0$, $\hat x$ is the unique unconstrained minimizer of $q(x)$. By Lemma~\ref{exact_underests}, since $Q \succ 0$, we have $\ell_{RS}(x) = q(x)$ for each $x \in F$, which implies that $\ell^* = q(x^*) = \ell_{RS}(x^*) = \ell^*_{RS}$. It follows that $(Q,c) \in \cE_{RS}$. The uniqueness follows from the strict convexity of $\ell_{RS}(\cdot)$ since $\ell_{RS}(x) = q(x)$ for each $x \in F$.   
\end{proof}
   
In the next section, we rely on duality theory to obtain a description of the set $\cE_{RS}$.

\subsection{The Dual Problem}

In this section, we present the dual of the SDP-RLT relaxation given by (RS) and establish several useful properties. 

Recall that the SDP-RLT relaxation is given by
\[
\textrm{(RS)} \quad \ell^*_{RS} = \min\limits_{(x,X) \in \R^n \times \cS^n} \left\{\frac{1}{2}\langle Q, X\rangle + c^T x: (x,X) \in \cF_{RS}\right\},
\]
where $\cF_{RS}$ is given by \eqref{def_cFRS}. 

By the Schur complement property, we have
\begin{equation} \label{schur_comp}
X - x x^T \succeq 0 \Longleftrightarrow \begin{bmatrix} 1 & x^T \\ x & X \end{bmatrix} \succeq 0,
\end{equation}
which implies that (RS) can be formulated as a linear semidefinite programming problem. 

By using the same set of dual variables $(u,v,W,Y,Z) \in \R^n \times \R^n \times \cS^n \times \R^{n \times n} \times \cS^n$ as in (R-D) corresponding to the common constraints in $\cF_R$ and $\cF_{RS}$ (see \eqref{def_cFR}), and defining the dual variable 
\[\begin{bmatrix} \beta & h^T \\ h & H \end{bmatrix} \in \cS^{n+1},\]
where $\beta \in \R$, $h \in \R^n$, and $H \in \cS^n$, corresponding to the additional semidefinite constraint in \eqref{schur_comp}, the dual problem of (RS) is given by 
\[
\begin{array}{llrcl}
\textrm{(RS-D)} & \max\limits_{(u,v,W,Y,Z,\beta,h,H) \in \R^n \times \R^n \times \cS^n \times \R^{n \times n} \times \cS^n \times \R \times \R^n \times \cS^n} & -e^T u - \frac{1}{2} e^T W e - \frac{1}{2} \beta& & \\
 & \textrm{s.t.} & & & \\
 & & -u + v - W e + Y^T e + h & = & c \\
 & & W - Y - Y^T + Z + H & = & Q \\
 & & u & \geq & 0 \\
 & & v & \geq & 0 \\
 & & W & \geq & 0 \\
 & & Y & \geq & 0 \\
 & & Z & \geq & 0 \\
 & & \begin{bmatrix} \beta & h^T \\ h & H \end{bmatrix} & \succeq & 0.
\end{array}
\]

In contrast with linear programming, stronger assumptions are needed to guarantee strong duality in semidefinite programming. We first establish that  (RS) and (RS-D) satisfy such assumptions.

\begin{lemma} \label{SDP-RLT-str-duality}
Strong duality holds between (RS) and (RS-D), and
optimal solutions are attained in both (RS) and (RS-D).    
\end{lemma}
\begin{proof}
Note that $\cF_{RS}$ is a nonempty and bounded set since $0 \leq x_j \leq 1$ and $X_{jj} \leq 1$ for each $j = 1,\ldots,n$. Therefore, the set of optimal solutions of (RS) is nonempty. Let $\hat x = \textstyle \frac{1}{2} e \in \R^n$ and let $\hat X = \hat x \hat x^T + \epsilon I \in \cS^n$, where $\epsilon \in (0,\textstyle\frac{1}{4})$. By Lemma~\ref{feas_cond}, $(\hat x, \hat X) \in \cF_{RS}$. Furthermore, it is a strictly feasible solution of (RS) since $(\hat x, \hat X)$ satisfies all the constraints strictly. Strong duality and attainment in (RS-D) follow from conic duality.
\end{proof}

Lemma~\ref{SDP-RLT-str-duality} allows us to give a complete characterization of optimality conditions for the pair (RS) and (RS-D).

\begin{lemma} \label{opt-cond-SDP-RLT}
 $(\hat x, \hat X) \in \cF_{RS}$ is an optimal solution of (RS) if and only if there exists 
 \[
 (\hat u,\hat v,\hat W,\hat Y,\hat Z,\hat \beta,\hat h,\hat H) \in \R^n \times \R^n \times \cS^n \times \R^{n \times n} \times \cS^n \times \R \times \R^n \times \cS^n
 \]
 such that
 \begin{eqnarray}
Q & = & \hat W - \hat Y - \hat Y^T + \hat Z + \hat H, \label{RLTSDPcond1}\\
c & = & -\hat u + \hat v - \hat W e + \hat Y^T e + \hat h, \label{RLTSDPcond2}\\
\begin{bmatrix} \hat \beta & \hat h^T \\ \hat h & \hat H \end{bmatrix} & \succeq & 0 \label{RLTSDPcond3}\\
\left\langle \begin{bmatrix} 1 & \hat x^T \\ \hat x & \hat X \end{bmatrix} , \begin{bmatrix} \hat \beta & \hat h^T \\ \hat h & \hat H \end{bmatrix} \right\rangle & = & 0, \label{RLTSDPcond4}
 \end{eqnarray}
 and \eqref{RLTcond1}--\eqref{RLTcond10} are satisfied.
\end{lemma}
\begin{proof}
The claim follows from strong duality between (RS) and (RS-D), which holds by Lemma~\ref{SDP-RLT-str-duality}.   
\end{proof}

Using Lemma~\ref{opt-cond-SDP-RLT}, we obtain the following description of the set of instances of (BoxQP) with an exact SDP-RLT relaxation.

\begin{proposition} \label{exact_SDP_RLT_char}
$(Q,c) \in \cE_{RS}$, where $\cE_{RS}$ is defined as in \eqref{exact_RLT_SDP_set}, if and only if there exists $\hat x \in F$ and there exists $(\hat u,\hat v,\hat W,\hat Y,\hat Z,\hat \beta,\hat h,\hat H) \in \R^n \times \R^n \times \cS^n \times \R^{n \times n} \times \cS^n \times \R \times \R^n \times \cS^n$ such that the conditions of Lemma~\ref{opt-cond-SDP-RLT} are satisfied, where $(\hat x, \hat X) = (\hat x, \hat x \hat x^T)$. Furthermore, in this case, $\hat x$ is an optimal solution of (BoxQP).
\end{proposition}
\begin{proof}
Suppose that $(Q,c) \in \cE_{RS}$. Let $\hat x \in F$ be an optimal solution of (BoxQP). By Corollary~\ref{implications}(ii), we obtain $\ell^*_{RS} = \ell^* = q(\hat x) = \ell_{RS}(\hat x)$. Therefore, $\hat x$ is an optimal solution of (RS1) given by \eqref{def_RS1}. Let $(\hat x, \hat X) = (\hat x, \hat x \hat x^T) \in \cF_{RS}$. We obtain $\textstyle \frac{1}{2}\langle Q, \hat X \rangle + c^T \hat x = q(\hat x) = \ell^*_{RS}$, which implies that $(\hat x, \hat x \hat x^T)$ is an optimal solution of (RS). The claim follows from Lemma~\ref{opt-cond-SDP-RLT}. 

For the reverse implication, note that $(\hat x, \hat X) = (\hat x, \hat x \hat x^T)$ is an optimal solution of (RS) by Lemma~\ref{opt-cond-SDP-RLT}. By a similar argument and using \eqref{convex_rels}, we obtain $\ell^* \leq q(\hat x) = \ell^*_{RS} \leq \ell^*$, which implies that $\ell^*_{RS} = \ell^*$, or equivalently, that $(Q,c) \in \cE_{RS}$. 

The second assertion follows directly from the previous arguments.
\end{proof}

In the next section, by relying on Proposition~\ref{exact_SDP_RLT_char}, we propose two algorithms to construct instances of (BoxQP) with different exactness guarantees. 

\subsection{Construction of Instances with Exact SDP-RLT Relaxations}

In this section, we present an algorithm for constructing instances of (BoxQP) with an exact SDP-RLT relaxation. Similar to Algorithm~\ref{Alg1}, Algorithm~\ref{Alg3} is based on designating $\hat x \in F$ and constructing an appropriate dual feasible solution that satisfies optimality conditions together with $(\hat x,\hat x \hat x^T) \in \cF_{RS}$.

\begin{algorithm}
\begin{algorithmic}[1]
\Require $n$, $\hat x \in F$
\Ensure $(Q,c) \in \cE_{RS}$
\State $\mathbb{L} \gets \mathbb{L}(\hat x)$, $\mathbb{B} \gets \mathbb{B}(\hat x)$, $\mathbb{U} \gets \mathbb{U}(\hat x)$
\State Choose an arbitrary $\hat u \in \R^n$ such that $\hat u_\mathbb{U} \geq 0$ and $\hat u_{\mathbb{L} \cup \mathbb{B}} = 0$.
\State Choose an arbitrary $\hat v \in \R^n$ such that $\hat v_\mathbb{L} \geq 0$ and $\hat v_{\mathbb{B} \cup \mathbb{U}} = 0$.
\State Choose an arbitrary $\hat W \in \cS^n$ such that $\hat W_{\mathbb{B}\cup \mathbb{L},\mathbb{B}\cup\mathbb{L}} = 0$ and $\hat W_{ij} \geq 0$ otherwise.
\State Choose an arbitrary $\hat Y \in \R^{n \times n}$ such that $\hat Y_{\mathbb{L}\mathbb{B}} = 0,~\hat Y_{\mathbb{L}\mathbb{U}} = 0,~\hat Y_{\mathbb{B}\mathbb{B}} = 0,~\hat Y_{\mathbb{B}\mathbb{U}} = 0$ and $\hat Y_{ij} \geq 0$ otherwise.
\State Choose an arbitrary $\hat Z \in \cS^n$ such that $\hat Z_{\mathbb{B} \cup \mathbb{U},\mathbb{B} \cup \mathbb{U}} = 0$
and $\hat Z_{ij} \geq 0$ otherwise.
\State Choose an arbitrary $\hat H \in \cS^n$ such that $\hat H \succeq 0$. 
\State $\hat h \gets - \hat H \hat x$, $\hat \beta \gets - \hat h^T \hat x$
\State $Q \gets \hat W - \hat Y - \hat Y^T + \hat Z + \hat H$
\State $c \gets -\hat u + \hat v - \hat W e + \hat Y^T e + \hat h$
\end{algorithmic}
\caption{(BoxQP) Instance with an Exact SDP-RLT Relaxation}
\label{Alg3}
\end{algorithm}

The next proposition establishes the correctness of Algorithm~\ref{Alg3}.

\begin{proposition} \label{Alg3-correct}
Algorithm~\ref{Alg3} returns $(Q,c) \in \cE_{RS}$, where $\cE_{RS}$ is defined as in \eqref{exact_RLT_SDP_set}. Conversely, any $(Q,c) \in \cE_{RS}$ can be generated by Algorithm~\ref{Alg3} with appropriate choices of $\hat x \in F$ and $(\hat u,\hat v,\hat W,\hat Y,\hat Z,\hat \beta,\hat h,\hat H) \in \R^n \times \R^n \times \cS^n \times \R^{n \times n} \times \cS^n \times \R \times \R^n \times \cS^n$.    
\end{proposition}
\begin{proof}
Since $\hat H \succeq 0$, it follows from Steps 8 and 9 of Algorithm~\ref{Alg3} that 
\begin{equation} \label{psd_rel}
\begin{bmatrix} \hat \beta & \hat h^T \\ \hat h & \hat H \end{bmatrix} = \begin{bmatrix} \hat x^T \\ -I \end{bmatrix} \hat H \begin{bmatrix} \hat x^T \\ -I \end{bmatrix}^T \succeq 0.
\end{equation}
Therefore, $(\hat u,\hat v,\hat W,\hat Y,\hat Z,\hat \beta,\hat h,\hat H) \in \R^n \times \R^n \times \cS^n \times \R^{n \times n} \times \cS^n \times \R \times \R^n \times \cS^n$ is a feasible solution of (RS-D). Furthermore, the identity in \eqref{psd_rel} also implies that
\[
\left\langle \begin{bmatrix} 1 & \hat x^T \\ \hat x & \hat x \hat x^T \end{bmatrix} , \begin{bmatrix} \hat \beta & \hat h^T \\ \hat h & \hat H \end{bmatrix} \right\rangle = \begin{bmatrix} 1 \\ \hat x \end{bmatrix}^T \begin{bmatrix} \hat \beta & \hat h^T \\ \hat h & \hat H \end{bmatrix} \begin{bmatrix} 1 \\ \hat x \end{bmatrix} = 0.
\]
It is easy to verify that the conditions of Lemma~\ref{opt-cond-SDP-RLT} are satisfied with $(\hat x, \hat X) = (\hat x,\hat x \hat x^T) \in \cF_{RS}$. Both assertions follow from Proposition~\ref{exact_SDP_RLT_char}.
\end{proof}

By Proposition~\ref{Alg3-correct}, we conclude that $\cE_{RS}$ is given by the union of infinitely many convex cones each of which can be represented by semidefinite and linear constraints.

Similar to Algorithm~\ref{Alg1}, we remark that Algorithm~\ref{Alg3} can be utilized to generate an instance of (BoxQP) with an exact SDP-RLT relaxation such that any designated feasible solution $\hat x \in F$ is an optimal solution of (BoxQP).

\subsection{Construction of Instances with Exact SDP-RLT and Inexact RLT Relaxations}

Recall that the SDP-RLT relaxation of any instance of (BoxQP) is at least as tight as the RLT relaxation.  In this section, we present another algorithm for constructing instances of (BoxQP) that admit an exact SDP-RLT relaxation but an inexact RLT relaxation, i.e., an instance in $\cE_{RS} \backslash \cE_R$ (cf.~\eqref{cERsubsCERS}). In particular, this algorithm can be used to construct instances of (BoxQP) such that the SDP-RLT relaxation not only strengthens the RLT relaxation, but also yields an exact lower bound. 

Note that Algorithm~\ref{Alg3} is capable of constructing all instances of (BoxQP) in the set $\cE_{RS}$. On the other hand, if one chooses $\hat x \in V$ and $\hat H = 0$ in Algorithm~\ref{Alg3}, which, in turn, would imply that $\hat h = 0$ and $\hat \beta = 0$, it is easy to verify that the choices of the remaining parameters satisfy the conditions of Algorithm~\ref{Alg1}, which implies that the resulting instance would already have an exact RLT relaxation, i.e., $(Q,c) \in \cE_R$. 

In this section, we present Algorithm~\ref{Alg4}, where we use a similar idea as in Algorithm~\ref{Alg2}, i.e., we aim to construct an instance of (BoxQP) such that $(\hat x,\hat x \hat x^T)$ is the unique optimal solution of (RS), where $\hat x \in F \backslash V$.

\begin{algorithm}
\begin{algorithmic}[1]
\Require $n$, $\hat x \in F \backslash V$
\Ensure $(Q,c) \in \cE_{RS} \backslash \cE_R$
\State $\mathbb{L} \gets \mathbb{L}(\hat x)$, $\mathbb{B} \gets \mathbb{B}(\hat x)$, $\mathbb{U} \gets \mathbb{U}(\hat x)$
\State Choose an arbitrary $\hat u \in \R^n$ such that $\hat u_\mathbb{U} \geq  0$ and $\hat u_{\mathbb{L} \cup \mathbb{B}} = 0$.
\State Choose an arbitrary $\hat v \in \R^n$ such that $\hat v_\mathbb{L} \geq 0$ and $\hat v_{\mathbb{B} \cup \mathbb{U}} = 0$.
\State Choose an arbitrary $\hat W \in \cS^n$ such that $\hat W_{\mathbb{B}\cup \mathbb{L},\mathbb{B}\cup\mathbb{L}} = 0$ and $\hat W_{ij} \geq 0$ otherwise.
\State Choose an arbitrary $\hat Y \in \R^{n \times n}$ such that $\hat Y_{\mathbb{L}\mathbb{B}} = 0,~\hat Y_{\mathbb{L}\mathbb{U}} = 0,~\hat Y_{\mathbb{B}\mathbb{B}} = 0,~\hat Y_{\mathbb{B}\mathbb{U}} = 0$ and $\hat Y_{ij} \geq 0$ otherwise.
\State Choose an arbitrary $\hat Z \in \cS^n$ such that $\hat Z_{\mathbb{B} \cup \mathbb{U},\mathbb{B} \cup \mathbb{U}} = 0$
and $\hat Z_{ij} \geq 0$ otherwise.
\State Choose an arbitrary $\hat H \in \cS^n$ such that $\hat H \succ 0$. 
\State $\hat h \gets - \hat H \hat x$, $\hat \beta \gets - \hat h^T \hat x$
\State $Q \gets \hat W - \hat Y - \hat Y^T + \hat Z + \hat H$
\State $c \gets -\hat u + \hat v - \hat W e + \hat Y^T e + \hat h$
\end{algorithmic}
\caption{(BoxQP) Instance with an Exact SDP-RLT Relaxation and an Inexact RLT Relaxation}
\label{Alg4}
\end{algorithm}

Note that Algorithm~\ref{Alg3} and Algorithm~\ref{Alg4} are almost identical, except that, in Step 7, we require that $\hat H \succ 0$ in Algorithm~\ref{Alg4} as opposed to $\hat H \succeq 0$ in Algorithm~\ref{Alg3}. The next result establishes that the output from Algorithm~\ref{Alg4} is an instance of (BoxQP) with an exact SDP-RLT but inexact RLT relaxation.

\begin{proposition} \label{Alg4correct}
Algorithm~\ref{Alg4} returns $(Q,c) \in \cE_{RS} \backslash \cE_R$, where $\cE_R$ and $\cE_{RS}$ are defined as in \eqref{exact_RLT_set} and \eqref{exact_RLT_SDP_set}, respectively.  
\end{proposition}
\begin{proof}
By the observation preceding the statement, it follows from Propositions~\ref{exact_SDP_RLT_char} and \ref{Alg3-correct} that $(Q,c) \in \cE_{RS}$ and that $(\hat x,\hat x \hat x^T)$ is an optimal solution of (RS). First, we show that this is the unique optimal solution of (RS). Suppose, for a contradiction, that there exists another optimal solution $(\tilde x, \tilde X) \in \cF_{RS}$. Note that, for any $A \succeq 0$ and $B \succeq 0$, $\langle A, B \rangle = 0$ holds if and only if $AB = 0$. Therefore, it follows from \eqref{RLTSDPcond4} that $\hat h - \hat H \tilde x = 0$. Since $\hat H \succ 0$, we obtain $\tilde x = \hat x$ by Step 8. By \eqref{psd_rel}, we obtain
\[
\left\langle \begin{bmatrix} 1 & \hat x^T \\ \hat x & \tilde X \end{bmatrix} , \begin{bmatrix} \hat \beta & \hat h^T \\ \hat h & \hat H \end{bmatrix} \right\rangle = \left\langle \begin{bmatrix} 1 & \hat x^T \\ \hat x & \tilde X \end{bmatrix} , \begin{bmatrix} \hat x^T \\ -I \end{bmatrix} \hat H \begin{bmatrix} \hat x^T \\ -I \end{bmatrix}^T \right\rangle = \langle \hat H, \hat X - \hat x \hat x^T \rangle = 0.
\]
Since $\hat H \succ 0$ by Step 7 and $ \hat X - \hat x \hat x^T \succeq 0$, it follows that $\tilde X = \hat x \hat x^T$, which contradicts our assumption. It follows that $(\hat x,\hat x \hat x^T)$ is the unique optimal solution of (RS), or equivalently, that $\hat x$ is the unique optimal solution of (RS1) given by \eqref{def_RS1}. By Proposition~\ref{exact_SDP_RLT_char} and \eqref{conv_underest_rel_1}, we conclude that $(Q,c) \in \cE_{RS}$ and that $\hat x \in F \backslash V$ is the unique optimal solution of (BoxQP). By Corollary~\ref{vertex-optimality}, $(Q,c) \not \in \cE_{R}$, which completes the proof.
\end{proof}

Algorithm~\ref{Alg4} can be used to construct an instance in the set $\cE_R \backslash \cE_{RS}$. In particular, it is worth noticing that the family of instances used in the proof of Lemma~\ref{SDP-RLT-opt-char} can be constructed by Algorithm~\ref{Alg4} by simply choosing $(\hat u, \hat v, \hat W, \hat Y, \hat Z) = (0,0,0,0,0)$. In particular, similar to Algorithm~\ref{Alg2}, it is worth noting that any instance constructed by Algorithm~\ref{Alg4} necessarily satisfies $Q_{kk} > 0$ for each $k \in \mathbb{B}$. On the other hand, recall that the SDP-RLT relaxation is always exact for $n \leq 2$. Therefore, similar to our discussion about Algorithm~\ref{Alg2}, we conclude that the set of instances that can be constructed by Algorithm~\ref{Alg4} may not necessarily encompass all instances in $\cE_R \backslash \cE_{RS}$. 

\section{Examples and Discussion} \label{Sec6}

In this section, we present numerical examples generated by each of the four algorithms given by Algorithms~\ref{Alg1}--\ref{Alg4}. 
We then close the section with a brief discussion.

\subsection{Examples}

In this section, we present instances of (BoxQP) generated by each of the four algorithms given by Algorithms~\ref{Alg1}--\ref{Alg4}. Our main goal is to demonstrate that our algorithms are capable of generating nontrivial instances of (BoxQP) with predetermined exactness or inexactness guarantees. 

\begin{example} \label{Ex1}
Let $n = 2$, $\mathbb{L} = \{1\}$, and $\mathbb{U} = \{2\}$ in Algorithm~\ref{Alg1}. Then, by Steps 2--6, we have
\[
\hat u = \begin{bmatrix} 0 \\ \alpha \end{bmatrix}, \quad \hat v = \begin{bmatrix} \beta \\ 0 \end{bmatrix}, \quad \hat W = \begin{bmatrix} 0 & \gamma \\ \gamma & \delta \end{bmatrix}, \quad \hat Y = \begin{bmatrix} \theta & 0 \\ \mu & \rho \end{bmatrix}, \quad \hat Z = \begin{bmatrix} \sigma & \eta \\ \eta & 0 \end{bmatrix},
\]
where each of $\alpha, \beta, \gamma, \delta, \theta, \mu, \rho, \sigma, \eta$ is a nonnegative real number. By Steps 7 and 8, we obtain 
\[
Q = \begin{bmatrix} \sigma - 2 \theta & \gamma + \eta - \mu \\ \gamma + \eta - \mu & \delta - 2 \rho \end{bmatrix}, \quad c = \begin{bmatrix} \beta + \theta + \mu - \gamma \\ \rho - \alpha - \gamma -\delta \end{bmatrix}.
\]
For instance, if we choose $\theta = \rho = \gamma = \eta = \mu = 0$, $\sigma + \delta > 0$, $\alpha \geq 0$, and $\beta \geq 0$, then $Q \succeq 0$, which implies that $q(x)$ is a convex function. If we choose $\sigma = \delta = \gamma = \eta = \mu = 0$, $\theta + \rho > 0$, $\alpha \geq 0$, and $\beta \geq 0$, then $-Q \succeq 0$, which implies that $q(x)$ is a concave function. Finally, if we choose $\theta = \delta = \gamma = \eta = \mu = 0$, $\sigma > 0$, $\rho > 0$, $\alpha \geq 0$, and $\beta \geq 0$, then $Q$ is indefinite, which implies that $q(x)$ is an indefinite quadratic function. For each of the three choices, the RLT relaxation is exact and $\hat x = \begin{bmatrix} 0 & 1 \end{bmatrix}^T$ is an optimal solution of the resulting instance of (BoxQP). Note that by setting $\hat H = 0$ and $\hat h = 0$ in Algorithm~\ref{Alg3}, the same observations carry over.
\end{example}

\begin{example} \label{Ex2}
Let $n = 3$, $\mathbb{L} = \{1\}$, $\mathbb{B} = \{2\}$, and $\mathbb{U} = \{3\}$ in Algorithm~\ref{Alg2}. Then, by Steps 3--8, we have $k = 2$ and 
\[
\hat u = \begin{bmatrix} 0 \\ 0 \\ \alpha \end{bmatrix}, \quad \hat v = \begin{bmatrix} \beta \\ 0 \\ 0 \end{bmatrix}, \quad \hat W = \begin{bmatrix} 0 & 0 & \gamma \\ 0 & \delta & \theta \\ \gamma & \theta & \mu \end{bmatrix}, \quad \hat Y = \begin{bmatrix} \rho & 0 & 0 \\ \sigma & 0 & 0 \\ \eta & \epsilon & \zeta \end{bmatrix}, \quad \hat Z = \begin{bmatrix} \kappa & \lambda & \nu \\ \lambda & \tau & 0 \\ \nu & 0 & 0 \end{bmatrix},
\]
where each of $\alpha, \beta, \gamma, \theta, \mu, \rho, \sigma, \eta, \epsilon, \zeta, \kappa, \lambda, \nu$ is a nonnegative real number, $\delta > 0$ and $\tau > 0$. By Steps 9 and 10, we obtain
\[
Q = \begin{bmatrix} \kappa - 2 \rho & \lambda - \sigma & \gamma + \nu - \eta \\ \lambda - \sigma & \delta + \tau & \theta - \epsilon\\ \gamma + \nu - \eta & \theta - \epsilon & \mu - 2 \zeta \end{bmatrix}, \quad c = \begin{bmatrix} \beta + \rho + \sigma + \eta - \gamma \\ \epsilon - \delta - \theta \\ \zeta - \gamma - \theta - \mu - \alpha \end{bmatrix}.
\]
If we set each of the parameters $\alpha, \beta, \gamma, \theta, \mu, \rho, \sigma, \eta, \epsilon, \zeta, \kappa, \lambda, \nu$ to zero, and choose any $\delta > 0$ and $\tau > 0$, then $Q \succeq 0$, which implies that $q(x)$ is a convex function. On the other hand, if we set each of the parameters $\alpha, \beta, \gamma, \theta, \mu, \sigma, \eta, \epsilon, \kappa, \lambda, \nu$ to zero, and choose any $\delta > 0$, $\tau > 0$ and $\rho + \zeta > 0$, then $Q$ is indefinite, which implies that $q(x)$ is an indefinite quadratic function. For each of the two choices, the RLT relaxation is inexact. Recall that an instance generated by Algorithm~\ref{Alg2} cannot have a concave objective function since $Q_{kk} = \hat W_{kk} + \hat Z_{kk} > 0$.  
\end{example}

\begin{example} \label{Ex4}
Let $n = 3$, $\mathbb{L} = \{1\}$, $\mathbb{B} = \{2\}$, and $\mathbb{U} = \{3\}$ in Algorithm~\ref{Alg4}. Then, by Steps 2--6, we have 
\[
\hat u = \begin{bmatrix} 0 \\ 0 \\ \alpha \end{bmatrix}, \quad \hat v = \begin{bmatrix} \beta \\ 0 \\ 0 \end{bmatrix}, \quad \hat W = \begin{bmatrix} 0 & 0 & \gamma \\ 0 & 0 & \theta \\ \gamma & \theta & \mu \end{bmatrix}, \quad \hat Y = \begin{bmatrix} \rho & 0 & 0 \\ \sigma & 0 & 0 \\ \eta & \epsilon & \zeta \end{bmatrix}, \quad \hat Z = \begin{bmatrix} \kappa & \lambda & \nu \\ \lambda & 0 & 0 \\ \nu & 0 & 0 \end{bmatrix},
\]
where each of $\alpha, \beta, \gamma, \theta, \mu, \rho, \sigma, \eta, \epsilon, \zeta, \kappa, \lambda, \nu$ is a nonnegative real number. By Step 7, $\hat H \succ 0$ is arbitrarily chosen. By Step 8, we have $\hat h = - \hat H \hat x$ and $\hat \beta = - \hat h^T \hat x$. By Steps 9 and 10, we therefore obtain
\[
Q = \begin{bmatrix} \kappa - 2 \rho & \lambda - \sigma & \gamma + \nu - \eta \\ \lambda - \sigma & 0 & \theta - \epsilon\\ \gamma + \nu - \eta & \theta - \epsilon & \mu - 2 \zeta \end{bmatrix} + \hat H, \quad c = \begin{bmatrix} \beta + \rho + \sigma + \eta - \gamma \\ \epsilon - \theta \\ \zeta - \gamma - \theta - \mu - \alpha \end{bmatrix} + \hat h.
\]
If we set each of the parameters $\alpha, \beta, \gamma, \theta, \mu, \rho, \sigma, \eta, \epsilon, \zeta, \kappa, \lambda, \nu$ to zero, then $Q = \hat H \succ 0$, which implies that $q(x)$ is a strictly convex function. On the other hand, if we set each of the parameters $\alpha, \beta, \gamma, \theta, \mu, \sigma, \eta, \epsilon, \kappa, \lambda, \nu$ to zero, and choose a sufficiently large $\rho + \zeta > 0$, then $Q$ is indefinite, which implies that $q(x)$ is an indefinite quadratic function. For each of the two choices, $\hat x$ is the unique optimal solution of the resulting instance of (BoxQP) and the SDP-RLT relaxation is exact whereas the RLT relaxation is inexact. Recall that an instance generated by Algorithm~\ref{Alg4} cannot have a concave objective function since $Q_{kk} = \hat H_{kk} > 0$. Indeed, such an instance of (BoxQP) necessarily has an optimal solution at a vertex whereas Algorithm~\ref{Alg4} ensures that the resulting instance of (BoxQP) has a unique solution $\hat x \in F \backslash V$. 
\end{example}

\subsection{Discussion}

We close this section with a discussion of the four algorithms given by Algorithms~\ref{Alg1}--\ref{Alg4}. Note that all instances of (BoxQP) can be divided into the following four sets:
\begin{eqnarray*} 
\cE_1 & = & \{(Q,c) \in \cS^n \times \R^n: \ell^*_R = \ell^*_{RS} = \ell^*\}, \\
\cE_2 & = & \{(Q,c) \in \cS^n \times \R^n: \ell^*_R <  \ell^*_{RS} = \ell^*\}, \\
\cE_3 & = & \{(Q,c) \in \cS^n \times \R^n: \ell^*_R =  \ell^*_{RS} < \ell^*\}, \\
\cE_4 & = & \{(Q,c) \in \cS^n \times \R^n: \ell^*_R <  \ell^*_{RS} < \ell^*\}.
\end{eqnarray*}

We clearly have $\cE_1 = \cE_R$, and any such instance can be constructed by Algorithm~\ref{Alg1}. On the other hand, Algorithm~\ref{Alg2} returns an instance in $\cE_2 \cup \cE_3 \cup \cE_4$. Any instance in $\cE_1 \cup \cE_2$ can be constructed by Algorithm~\ref{Alg3}. Finally, Algorithm~\ref{Alg4} outputs an instance in the set $\cE_2 = \cE_{RS} \backslash \cE_R$. 

Note that one can generate a specific instance of (BoxQP) with an inexact SDP-RLT relaxation by extending the example in  Section~\ref{const_inexact_rlt}~\cite{anstreicher2010computable}. Let $n = 2k + 1 \geq 3$ and consider the instance $(Q,c) \in \cS^n \times \R^n$ given by $Q = \textstyle\frac{1}{n} ee^T - I$, where $I \in \cS^n$ denotes the identity matrix, and $c = 0$. Since $Q$ is negative semidefinite, the optimal solution of (BoxQP) is attained at one of the vertices. It is easy to verify that any vertex $v \in F$ with $k$ (or $k + 1$) components equal to 1 and the remaining ones equal to zero is an optimal solution, which implies that $\ell^* = \textstyle\frac{1}{2}\left(\textstyle\frac{k^2}{n} - k\right)$. Let $\hat x = \textstyle\frac{1}{2} e$ and 
\[
\hat X = \hat x \hat x^T + \hat M = \frac{1}{4} ee^T + \frac{1}{4(n-1)} \left(n I - ee^T \right) = \frac{1}{4} \left(1 + \frac{1}{n-1}\right) I + \frac{1}{4}\left(1 - \frac{1}{n-1}\right) ee^T.
\]

It is easy to verify that $(\hat x, \hat X) \in \cF_{RS}$. Therefore, 
\[
\ell^*_{RS} \leq \frac{1}{2} \langle Q, \hat X \rangle + c^T \hat x = -\frac{n}{8}.
\]
Using $n = 2k + 1$, we conclude that $\ell^*_{RS} < \ell^*$, i.e., the SDP-RLT relaxation is inexact. Finally, this example can be extended to an even dimension $n = 2k \geq 4$ by simply constructing the same example corresponding to $n = 2k - 1$ and then adding a component of zero to each of $\hat x$ and $c$, and adding a column and row of zeros to each of $Q$ and $\hat X$. 

An interesting question is whether an algorithm can be developed for generating more general instances with inexact SDP-RLT relaxations, i.e., the set of instances given by $\cE_3 \cup \cE_4$. One possible approach is to use a similar idea as in Algorithms~\ref{Alg2} and~\ref{Alg4}, i.e., designate an optimal solution $(\hat x,\hat X) \in \cF_{RS}$, which is not in the form of $(v,vv^T)$ for any vertex $v \in F$, and identify the conditions on the other parameters so as to guarantee that $(\hat x,\hat X)$ is the unique optimal solution of the SDP-RLT relaxation (RS). Note that Lemma~\ref{opt-cond-SDP-RLT} can be used to easily construct an instance of (BoxQP) such that any feasible solution $(\hat x,\hat X) \in \cF_{RS}$ is an optimal solution of (RS). In particular, the condition \eqref{RLTSDPcond4} can be satisfied by simply choosing an arbitrary matrix $B \in\cS^k$ such that $B \succeq 0$, and by defining 
\[
\begin{bmatrix} \hat \beta & \hat h^T \\ \hat h & \hat H \end{bmatrix} = P B P^T,
\]
where $P \in \R^{(n+1) \times k}$ is a matrix whose columns form a basis for the nullspace of the matrix 
\[
\begin{bmatrix} 1 & \hat x^T \\ \hat x & \hat X \end{bmatrix}.
\]
For instance, the columns of $P$ can be chosen to be the set of eigenvectors corresponding to zero eigenvalues. However, this procedure does not necessarily guarantee that $(\hat x,\hat X) \in \cF_{RS}$ is the unique optimal solution of (RS). Therefore, a characterization of the extreme points and the facial structure of $\cF_{RS}$ may shed light on the algorithmic construction of such instances. We intend to investigate this direction in the near future. 

\section{Concluding Remarks} \label{Sec7}

In this paper, we considered RLT and SDP-RLT relaxations of quadratic programs with box constraints. We presented algebraic descriptions of instances of (BoxQP) that admit exact RLT relaxations as well as those that admit exact SDP-RLT relaxations. Using these descriptions, we proposed four algorithms for efficiently constructing an instance of (BoxQP) with predetermined exactness or inexactness guarantees. In particular, we remark that Algorithms~\ref{Alg1}, \ref{Alg3}, and \ref{Alg4} can be used to construct an instance of (BoxQP) with a known optimal solution, which may be of independent interest for computational purposes.

In the near future, we intend to investigate the facial structure of the feasible region of the SDP-RLT relaxation and exploit it to develop algorithms for generating instances of (BoxQP) with an inexact SDP-RLT relaxation. 

Another interesting direction is the computational complexity of determining whether, for a given instance of (BoxQP), the RLT or the SDP-RLT relaxation is exact. Our algebraic descriptions do not yield an efficient procedure for this problem. An efficient recognition algorithm may have significant implications for extending the reach of global solvers for (BoxQP).

\bibliographystyle{abbrv}
\bibliography{sn-bibliography}

\end{document}